\documentclass[a4paper,11pt]{article}
\pdfoutput=1
\usepackage{comment}
\usepackage[utf8]{inputenc}
\usepackage{lmodern}
\usepackage{amssymb}
\usepackage{amsfonts}
\usepackage{amsthm}
\usepackage{amsmath}
\usepackage{mathtools}
\usepackage{mathrsfs}
\numberwithin{equation}{section}
\usepackage{fancyvrb}
\usepackage{graphicx}
\usepackage[square,sort,comma,numbers]{natbib}
\usepackage[nottoc]{tocbibind}
\usepackage{paralist}
\usepackage[usenames]{xcolor}
\usepackage{color}
\usepackage{soul}
\usepackage[hyperindex,pageanchor,raiselinks,bookmarks,pdftex,unicode]{hyperref}
\hypersetup{breaklinks=true}
\hypersetup{unicode}
\include{resource_macros}

\begin{document}
\title{A bending-torsion theory for thin and ultrathin rods as a $\Gamma$-limit of atomistic models}
\author{Bernd Schmidt \\ \textsc{\footnotesize{Institut f\"{u}r Mathematik, Universit\"{a}t Augsburg, D-86135 Augsburg, Germany}} \\ \footnotesize{Email address: \href{mailto:bernd.schmidt@math.uni-augsburg.de}{\texttt{bernd.schmidt@math.uni-augsburg.de}}} \\ 
   \and Ji\v{r}\'{\i} Zeman \\ \textsc{\footnotesize{Institut f\"{u}r Mathematik, Universit\"{a}t Augsburg, D-86135 Augsburg, Germany}} \\ \footnotesize{Email address: \href{mailto:geozem@seznam.cz}{\texttt{geozem@seznam.cz}}}}
\date{\today}
\maketitle
\begin{abstract}
The purpose of this note is to establish two continuum theories for the bending and torsion of inextensible rods as $\Gamma$-limits of 3D atomistic models. In our derivation we study simultaneous limits of vanishing rod thickness $h$ and interatomic distance $\e$. First, we set up a novel theory for \textit{ultrathin rods} composed of finitely many atomic fibres ($\e\sim h$), which incorporates surface energy and new discrete terms in the limiting functional. This can be thought of as a contribution to the mechanical modelling of nanowires. Second, we treat the case where $\e\ll h$ and recover a  nonlinear rod model -- the modern version of Kirchhoff's rod theory.


\end{abstract}
\section{Introduction}
Since the first boom in research on carbon nanotubes in the 1990s, we have been experiencing discoveries of a wide variety of 1D nanomaterials. These include nanowires, nanorods, nanopillars, and nanowhiskers \cite{nanoMod, nanopillars}, which find applications in electronics, photonics \cite{nanotech, NWsolar}, sensor design \cite{NWsense,ZnOsense} or biomedicine \cite{SiNW,nanowhiskers}.

As such thin structures only have tens of nanometres in diameter, they exhibit unusual deformation behaviour under external loads (e.g. great flexibility, anisotropy or surface effects). Despite the fast-paced progress, loading experiments remain challenging due to the need of specialized and highly precise measurement devices, so advances in mechanical modelling and computational studies of nanomaterials are still very desirable.

Elastic theories for one-dimensional rods or beams witness a long history (see \cite{Antman,OReilly} for an overview). An early milestone was marked in \cite{Kirchhoff} in 1859 and since then, Kirchhoff's rod theory has become the most widespread one for describing slender elastic bodies moving in 3D space (although a reformulation in modern notation is now used). The  elastic energy of an isotropic Kirchhoff rod with length $L$ can be expressed as
$$\mathcal{E}(y,d_2,d_3):=\frac{1}{2}\int_0^L E(I_2\kappa_2^2+I_3\kappa_3^2)+J\tau^2\md \x_1,$$
where $y\colon (0,L)\goto\R^3$ is the deformation of the rod and $d_2\colon (0,L)\goto\R^3$ and $d_3\colon (0,L)\goto\R^3$ are the so-called directors, which form an orthonormal frame ($\pl_{\x_1} y$,$d_2$,$d_3$) moving along $\x_1$.  The scalars $\kappa_2=\pl_{\x_1}^2 y\cdot d_2$ and $\kappa_3=\pl_{\x_1}^2 y\cdot d_3$ are called curvatures and $\tau=\pl_{\x_1} d_2\cdot d_3$ is the torsion. Young's modulus is denoted by $E$ and torsional rigidity by $J$ -- the latter is calculated for the shape of the cross section $S\subset\R^2$ of the rod. The second moments of area are $I_s=\int_S \x_s^2\md\x_2\md\x_3$ with $s=2$ or $s=3$.

In \cite{MMh4}, a nonlinear bending-torsion theory for inextensible rods was rigorously derived from three-dimensional elasticity using $\Gamma$-convergence. This theory, which was also independently obtained in \cite{Pantz}, embraces Kirchhoff rods as a special case. We refer to \cite{BraiBeg,BraiHand} for an introduction to $\Gamma$-convergence and to \cite{strings,membranes,FrM02,MMh6,Scardia,FrM06,Timoshenko,incompRods} for other results on mathematical derivation of dimensionally reduced theories in elasticity.

For the purposes of identification of Young's modulus and Poisson's ratio, Kirchhoff's rod theory has already been applied to nanowires.  \cite{KirchNW} However, the natural question arises whether atomistic effects should not be part of continuum theories for bodies which only consist of a few atomic layers in their transversal direction. Bearing this in mind, Friesecke and James proposed in \cite{FJ00} a method for deriving continuum models of 2D and 1D nanomaterials when in-plane strain is dominant (\textit{membrane theory}) and the approach was implemented rigorously in \cite{BS08} for thin films. The work \cite{BS06} focused on the bending of \textit{Kirchhoff's plates} and introduced a continuum theory for thin films which comprise no more than several layers. A similar derivation of \textit{von-Kármán's plate theory} has only been achieved recently \cite{BrS19}. To complete the survey of research on  microscopic origins of elasticity theory, we point the reader to \cite{Blanc,AC,Conti,EMing,BS09,BrS13,BrS16,Bach,ALP21} and the references therein.

In the present article, we treat continuum limits of discrete energies of the type
\begin{equation*}
E^{(k)}(y^{(k)})=\sum_{x\in\Lambda_{\e_k}'}W_{\rm cell}\bigl(\vec{y}^{\,(k)}(x)\bigr)+\text{surface terms},
\end{equation*}
where $\Lambda_{\e_k}'$ is an $\e_k$-fine cubic crystalline lattice in the shape of a thin rod, $y^{(k)}$ its deformation and the matrix $\vec{y}^{\,(k)}(x)$ describes the deformation of an atomic cube around the point $x$. Such cell energies $W_{\rm cell}$ cover the case of nearest neighbour and next-to-nearest neighbour interactions and appeared previously e.g. in \cite{ValFail,Conti,BS06}. 

Section 2 sets up basic notation and introduces  model assumptions that are common for the rest of this article. We also formulate a compactness theorem that complements theorems on $\Gamma$-convergence in the following sections. Having appeared in \cite{MMh4}, the result needs only minor adjustments in our discrete framework. 

\begin{figure}[h]
  \caption{An illustration of the simultaneous dimension reduction and discrete-to-continuum limit.}\label{fig:dimR}
  \centering
\includegraphics[width=7.5cm]{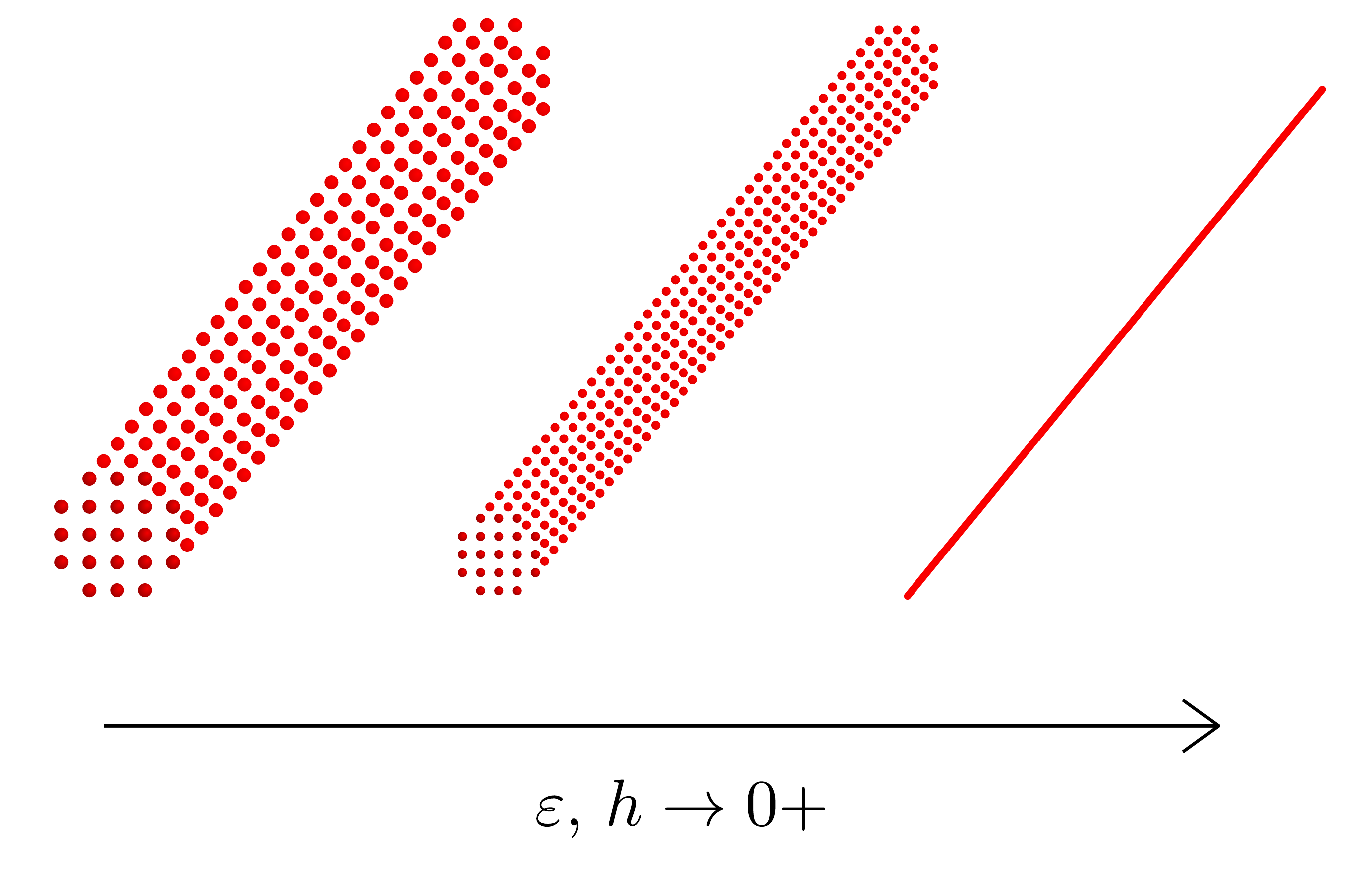}
\end{figure}
We seek a limiting energy functional $E_{\rm lim}$ for the continuum model. To get a nontrivial limit with $k\goto\infty$, we multiply the energy $E^{(k)}$ by the volume element $\e_k^3$ and divide it by the fourth power of the rod thickness $h_k$, which is the energy scaling corresponding to bending and torsion without extending the rod, cf. \cite{MMh4}.

We are interested in two possible limit processes, which yield different effective models in the end (see Figure \ref{fig:dimR} for an illustration).
\begin{enumerate}
\item To model an \textit{ultrathin rod} composed of a small number of atomic fibres, we let the interatomic distance $\e_k\goto 0+$ and keep $h_k/\e_k$ fixed. This is the content of Section~\ref{sec:U}, which includes the $\Gamma$-convergence Theorem \ref{GammaU} -- our main contribution. Remarkably, even though this new \textit{bending/torsion theory for ultrathin rods} thus derived can be related to the findings in \cite{MMh4}, our elastic energy functional features a so-called \textit{ultrathin correction} and surface terms, none of which would be present in a limiting theory based on the Cauchy-Born rule. Moreover, in the limiting functional we identify a discrete minimization formula accounting for warping the rod's cross section –  a more complex ingredient than in  plate theories from \cite{BS06} and \cite{BrS19}. With these traits, we believe that our proposed effective model might describe very thin 1D nanostructures more accurately than would conventional elasticity. 
\item When the numbers of atoms in the rod in the directions $x_1$, $x_2$, $x_3$ are large, we speak of a \textit{thin rod} and study the simultaneous limit with $\e_k\goto 0+$ and $h_k\goto 0+$ in such a way that $\frac{h_k}{\e_k}\goto\infty$. In this regime, which we investigate in Section~\ref{sec:T}, all discreteness fades away and we recover the continuum functional from \cite{MMh4} (see Theorem \ref{GammaT}).
\end{enumerate}

In our forthcoming paper \cite{fracRods} we will extend the results about ultrathin rods to brittle materials. Other approaches to nanowire mechanical modelling include \cite{genCB,quasiCont,MS15} and \cite{heliCB,torsVib}. Several works have also used couple-stress theories to account for size effects in Kirchhoff rods. \cite{coupStress}

For ease of notation, we only consider $\e_k:=1/k$ in the following, but it would also be possible to work with arbitrary interatomic distances, see \cite{BrS19}.

\section{Notation, common model assumptions}\label{sec:not}

\subsection{Basics}

If $S\subset\R^n$, we write $|S|$ for the $n$-dimensional Lebesgue measure of $S$. In the whole text, we reserve the letter $C$ for a generic positive constant whose value may vary from line to line, but is independent of the quantities involved in a limit passage. We use standard notation for function spaces: namely the Lebesgue spaces $L^p(\O;\R^n)$, $p\in[1,\infty]$, Sobolev spaces $H^m(\O;\R^n)=W^{m,2}(\O;\R^n)$, $m\in\N$, and weak convergence ($f_k\weakto f$). Further, $A_{\bullet j}$ denotes the $j$-th column vector of a matrix $A\in\R^{m\times n}$; $\R^{3\times 3}_{\rm skew}$ stands for the space of all 3-by-3 skew-symmetric matrices; $e_i=\Id_{\bullet i}$, $1\leq i\leq 3$, are the standard basis vectors in $\R^3$, and $|u|$ and $|A|=\sqrt{\mathrm{Tr}\,\T{A}A}$ denote the Euclidean and Frobenius norms of $u\in\R^n$ and $A\in\R^{m\times n}$, respectively. \tc{All vectors, unless otherwise specified, are treated as column vectors.} For an open set $\O\subset\R^n$ we write $\O'\subset\subset\O$ if $\bar{\O}'\subset\O$ and $\bar{\O}'$ is compact. Finally, $V^\bot$ is the orthogonal complement of a subspace $V$ in an inner product space $X$.

\subsection{Discrete model}\label{sec:DiscreteModel}
Our starting point is an atomistic interaction model for an elastic rod. We consider a cubic atomic lattice $\Lambda_k$, \tc{given by}
$$\Lambda_k=\Bigl([0,L]\times\frac{1}{k}\overline{S_k}\Bigr)\cap\frac{1}{k}\Z^3,$$
where $\frac{1}{k}$ is the interatomic spacing and $L>0$ denotes the length of the rod. Its cross section is the polygonal set $\emptyset\neq S_k\subset\R^2$ (possibly not simply connected) determining a cross-sectional lattice $\mathcal{L}_k:=\overline{S_k}\cap \Z^2$ and for which there is a set $\mathcal{L}'_k\subset(\frac{1}{2}+\Z)^2$ such that
\begin{align}\label{eq:Sk-def}
  S_k=\mathrm{Int}\,\bigcup_{\xb'\in\mathcal{L}'_k}\Bigl(\xb'+\Bigl[-\frac{1}{2},\frac{1}{2}\Bigr]^2\Bigr).
\end{align}
(It is assumed that $\xb'\in\mathcal{L}_k'$ whenever $\xb'+\{-\frac{1}{2},\frac{1}{2}\}^2\subset\mathcal{L}_k$.) If $S_k=S$ is a fixed cross section that does not depend on $k$ we will speak of an {\em ultrathin} rod. \tc{The rod's thickness} is then comparable to the typical interatomic spacing. \tc{By contrast, in a {\em thin} rod the scaled cross section $\frac{1}{k}S_k$} eventually exhausts a domain of diameter $h$, where $\frac{1}{k} \ll h \ll 1$. We use the symbol $\Lambda_k'$ for the lattice of midpoints of open cubes with sidelength $1/k$ and corners in $\Lambda_k$. 

These set-ups may be described simultaneously by our fixing a positive null sequence $(h_k)$ with $h_k \ge 1/k$ that we choose as equal to $1/k$ in the ultrathin case and for which we suppose $kh_k \to \infty$ for merely thin rods. We then assume that there exists a fixed bounded Lipschitz domain $S \subset \R^2$ such that the above $S_k$ is the unique largest (in terms of cardinality) connected set of the form \eqref{eq:Sk-def} that is contained in $kh_k S$.

The lattice $\Lambda_k$ corresponds to an undeformed reference configuration that is subject to a static deformation $y^{(k)}\colon \Lambda_k\goto\R^3$, which stores elastic energy into the rod. As the energy originates from interactions of nearby atoms we introduce a rescaling to atomic units by passing to a rescaled lattice with unit distances between atoms. 

\tc{Points in this lattice are distinguished using the hat diacritic  -- here for $x=(x_1, x_2, x_3)\in\R^3$ we write $\hat{x}_1:=kx_1$, $\hat{x}'=(\hat{x}_2,\hat{x}_3):=kx'=k(x_2,x_3)$} and $\hat{y}^{(k)}(\hat{x}_1,\hat{x}_2,\hat{x}_3):=k y^{(k)}(\frac{1}{k}\hat{x}_1,\frac{1}{k}\hat{x}')$ so that $\hat{y}^{(k)}\colon k\Lambda_k\to\R^3$. Then $\hat{\Lambda}_k$, $\hat{\Lambda}'_k$ stand for the sets of all $\hat{x}=(\hat{x}_1,\hat{x}_2,\hat{x}_3)$ such that the corresponding downscaled points $x$ lie in the lattices $\Lambda_k$, $\Lambda_k'$, respectively. We introduce eight direction vectors $\zf^1,\dots,\zf^8$:

\begin{center}
\begin{tabular}{l l}
$\zf^1 = \frac{1}{2}\T{(-1,-1,-1)}$, & $\zf^5 = \frac{1}{2}\T{(+1,-1,-1)},$\\
$\zf^2 = \frac{1}{2}\T{(-1,-1,+1)}$, & $\zf^6 = \frac{1}{2}\T{(+1,-1,+1)},$\\
$\zf^3 = \frac{1}{2}\T{(-1,+1,+1)}$, & $\zf^7 = \frac{1}{2}\T{(+1,+1,+1)},$\\
$\zf^4 = \frac{1}{2}\T{(-1,+1,-1)}$, & $\zf^8 = \frac{1}{2}\T{(+1,+1,-1)}$.
\end{tabular}
\end{center}
This allows us to collect into a matrix the information about the deformation of a unit cell $\hat{x}+\{-\frac{1}{2},\frac{1}{2}\}^3$, $\hat{x}\in\hat{\Lambda}'_k$: 
$$\vec{y}^{\,(k)}(\hat{x})=(\hat{y}^{(k)}(\hat{x}+\zf^1)|\cdots|\hat{y}^{(k)}(\hat{x}+\zf^8))\in\R^{3\times 8}.$$ 
With $\langle \hat{y}^{(k)}(\hat{x})\rangle=\frac{1}{8}\sum_{i=1}^8 \hat{y}^{(k)}(\hat{x}+\zf^i)$, $\hat{x}\in\hat{\Lambda}'_k$ we further define the discrete gradient 
\tc{$$\bar{\nabla}\hat{y}^{(k)}(\hat{x})=\vec{y}^{\,(k)}(\hat{x})-\langle \hat{y}^{(k)}(\hat{x})\rangle(1,\ldots,1)\in\R^{3\times 8}.$$}
Then the matrix $\bar{\Id}=(\zf^1|\cdots|\zf^8)\in\R^{3\times 8}$ is the discrete gradient of $\hat{y}^{(k)}=\mathrm{id}$. Note that a discrete gradient has the sum of columns equal to $0$.

There are two more important subsets of $\R^{3\times 8}$:
$$\bar{\rm SO}(3):=\{R\,\bar{\Id};\;R\in\mathrm{SO}(3)\},\quad V_0:=\{(c|\cdots|c)\in\R^{3\times 8};\;c\in\R^3\}.$$

\subsection{Rescaling, interpolation and extension}\label{sec:inter}
It is desirable to have the deformations defined on a common domain $\O:=(0,L)\times S$, independent of $k$, in order to handle their convergence. \tc{Given a positive null sequence $(h_k)$ such that $h_k \ge 1/k$ (and  $h_k=1/k$} in the ultrathin case) set $\yk(\x_1,\x_2,\x_3):=y^{(k)}(\x_1,h_k\x')$ for $(\x_1,h_k\x')\in\Lambda_k$. Furthermore, we introduce an interpolation of $\ybk$ so that it is also defined outside lattice points.

Let $\zfb^i=(\frac{1}{k}\zf_1^i,\frac{1}{kh_k}\zf_2^i,\frac{1}{kh_k}\zf_3^i)$ and $\tilde{\Lambda}_k'=\{\xi\in\R^3;\;(k\xi_1,kh_k\xi')\in\hat{\Lambda}_k'\}$. We split every block $\Qb(\bar{\xb})=\bar{\xb}+[-\frac{1}{2k},\frac{1}{2k}]\times[-\frac{1}{2kh_k},\frac{1}{2kh_k}]^2$, $\bar{\xb}\in\tilde{\Lambda}_k'$, into 24 simplices as in \cite{BS06,BrS19} and get a piecewise affine interpolation of $\ybk$, which we denote again by $\ybk$. More precisely, set $\ybk(\bar{\xb}):=\frac{1}{8}\sum_{i=1}^8\ybk(\bar{\xb}+\zfb^i)$ and for each face $\tilde{F}$ of the block $\Qb(\bar{\xb})$ and the corresponding centre $\xb_{\tilde{F}}$ of the face $\tilde{F}$, define $\ybk(\xb_{\tilde{F}}):=\frac{1}{4}\sum_j\ybk(\bar{\xb}+\zfb^j)$, where we sum over all $j$ such that $\bar{\xb}+\zfb^j$ is a corner of $\tilde{F}$. In fact, a face can be labelled as $\tilde{F}_{ij}$ if it has $\bar{\xb}+\zfb^i$ and $\bar{\xb}+\zfb^j$ such that $|\zf^i-\zf^j|=1$ as vertices; the ambiguity in this notation can be resolved by using the order of indices. Then, let $\ybk$ be interpolated in an affine way on every $\tilde{T}_{ij}=\mathrm{conv}\{\bar{\xb},\bar{\xb}+\zfb^i,\bar{\xb}+\zfb^j,\xb^{ij}\}$ with $\xb^{ij}$ being the centre of the face $\tilde{F}_{ij}$, so that $\ybk$ is everywhere continuous.

We thus obtain $\ybk\colon [0,L_k]\times\tc{\bar{S}_k}\to\R^3$, where we have abbreviated $L_k := \lfloor kL\rfloor/k$. It satisfies
\begin{equation}\label{eq:surfVolMean}
\ybk(\xb_{\tilde{F}})=\dashint_{\tilde{F}}\ybk\md \mathcal{H}^2,\quad\ybk(\bar{\xb})=\dashint_{\Qb(\xb)}\ybk(\xi)\md\xi
\end{equation}
for any face $\tilde{F}$ of $\Qb(\xb)$ with face centre $\xb_{\tilde{F}}$.

\tc{Setting $\nabla_k\ybk:=\bigl(\frac{\pl \yk}{\pl \x_1}\,\Big|\,h_k^{-1}\frac{\pl \yk}{\pl \x_2}\,\Big|\,h_k^{-1}\frac{\pl \yk}{\partial \x_3}\bigr)$, we proceed with an auxiliary result.}

\begin{lemma}\label{equivNorms}
There are $c,C>0$ such that for any $k\in\N$, $h_k>0$ and lattice block $\Qb(\bar{\x})=\bar{\x}+[-\frac{1}{2k},\frac{1}{2k}]\times[-\frac{1}{2kh_k},\frac{1}{2kh_k}]^2$
with centre $\bar{\x}\in\tilde{\Lambda}_k'$ and \tc{corresponding $\hat{x}=(k\bar{\x}_1,kh_k\bar{\x}')\in\hat{\Lambda}_k'$, 
\begin{align}\label{eq:norms}
  c|\bar{\nabla} \hat{y}^{(k)}(\hat{x})|^2\leq k^3h_k^2\int_{\tilde{Q}(\bar{\x})} |\nabla_k \yk|^2\md\xi\leq C|\bar{\nabla} \hat{y}^{(k)}(\hat{x})|^2. 
\end{align}}
\end{lemma}
\begin{proof}
The statement is contained in \cite[Lemma~3.5]{BS09}.
\end{proof}

We now construct an extension to `ghost atoms' in a tubular neighbourhood of the rod whose rigidity is controlled by the original atom positions. For $m\in\N$ set 
\begin{align*}
\mathcal{L}^{\rm ext}_k
&=\mathcal{L}_k+\{-m,\ldots,m\}^2,
& 
\Lkex
&=\{-\tfrac{1}{k},0,\ldots,L_k+\tfrac{1}{k}\}\times \tfrac{1}{k}\mathcal{L}^{\rm ext}_k,
\\
\mathcal{L}'^{,\rm ext}_k
&=\mathcal{L}'_k+\{-m,\ldots,m\}^2
&
\Lkexc
&=\{-\tfrac{1}{2k},\tfrac{1}{2k},\ldots,L_k+\tfrac{1}{2k}\}\times \tfrac{1}{k}\mathcal{L}'^{,\rm ext}\\ 
S^{\rm ext}_k
&=S_k+(-m,m)^2,
&
\Omega^{\rm ext}_k
&= (-\tfrac{1}{k}, L_k+\tfrac{1}{k}) \times \frac{1}{kh_k} S^{\rm ext}_k.
\end{align*}
We suppress $m$, which will be a fixed constant, from our notation. It will be equal to $1$ for ultrathin rods and $\ge 1$ such that $S^{\rm ext}_k \supset kh_k S$ for thin rods. We also consider the lattices $\Lkexb$ and $\Lkexcb$ that are related to their unrescaled versions $\Lkex$ and $\Lkexc$ like we saw it for $\tilde{\Lambda}_k'$ above. 

Our extension follows a scheme from \cite[Section~3.1]{BS09}, see in particular \cite[Lemmas~3.1, 3.2 and~3.4]{BS09} and cf.\ also \cite[Lemma~3.1]{BrS19}. Notice that for our choice of $S_k$ as the largest connected set of the form \eqref{eq:Sk-def} that is contained in $kh_k S$ for a bounded Lipschitz domain $S \subset \R^2$ in particular guarantees that there is a constant $C > 0$, independent of $k$, such that for any two points \tc{$\hat{x}',\hat{y}' \in \mathcal{L}_k'$
$$ \mathrm{dist}_{\mathcal{L}_k'}(\hat{x}',\hat{y}') 
   \le C|\hat{x}'-\hat{y}'|, $$ 
where 
$$ \mathrm{dist}_{\mathcal{L}_k'}(\hat{x}',\hat{y}') 
   = \min \bigl\{ N \in \N_0 : \exists\, \hat{x}'=\hat{x}'_0, \ldots, \hat{x}'_N=\hat{y}' \in \mathcal{L}_k' \text{ with }  |\hat{x}_{n+1}-\hat{x}_n|=1\, \forall\, n < N \bigr\}
$$
denotes the lattice geodesic distance of two elements $\hat{x}',\hat{y}' \in \mathcal{L}_k'$.}

\begin{lemma}\label{lemma:ext-general}
There are extensions $y^{(k)}\colon \Lkex \to \R^3$ such that their interpolations $\ybk$ satisfy
\begin{align*}
\operatorname{ess\, sup}_{\Omega^{\rm ext}_k} \mathrm{dist}^2(\nabla_k \ybk,\mathrm{SO}(3)) 
&\le C \operatorname{ess\, sup}_{(0,L_k)\times \frac{1}{kh_k} S_k} \mathrm{dist}^2(\nabla_k \ybk,\mathrm{SO}(3)) 
\shortintertext{and}
\int_{\Omega^{\rm ext}_k} \mathrm{dist}^2(\nabla_k \ybk,\mathrm{SO}(3)) \md x 
&\le C \int_{(0,L_k)\times \frac{1}{kh_k} S_k} \mathrm{dist}^2(\nabla_k \ybk,\mathrm{SO}(3)) \md x. 
\end{align*}
\end{lemma}
\begin{proof}
Let $y^{(k)}\colon\Lambda_k \to \R^3$ be a lattice deformation. We partition $\Lkexc \setminus \Lambda_k'$ into the $8$ sublattices $\Lambda_{k,i}' = (\Lkexc \setminus \Lambda_k') \cap \frac{1}{k} ( \zf^i + 2\Z^3 )$ and apply the following extension procedure consecutively for $i = 1, \ldots, 8$:  

If $x \in \Lambda_{k,i}'$ we write $\mathcal{B}_R(x)$ for the set of those $z\in\Lkexc$ with $|z-x| \le R/k$ for which $y^{(k)}(z+\frac{1}{k}\zf^j)$ is defined already for all $1 \le j \le 8$. Now if $\mathcal{B}_1(x) \ne \emptyset$, extend $y^{(k)}$ to all $z+\frac{1}{k}\zf^j$, $1 \le j \le 8$, by choosing an extension such that ${\rm dist}^2(\bar{\nabla}\hat{y}^{(k)}(\hat{x}), {\rm SO}(3)\bar{\Id})$ is minimal. 

Due to \cite[Lemma~3.1]{BS09} and \tc{the property} of lattice geodesics within $\mathcal{L}_k'$, this distance will then be controlled by 
\[ C \sum_{z\in\mathcal{B}_R(x)} {\rm dist}^2(\bar{\nabla}\hat{y}^{(k)}(\hat{z}), \bar{\rm SO}(3)), \] 
for some uniformly bounded $R$. We repeat this extension step $8m$ times. 
\end{proof}

\begin{rem}
The construction implies that for ultrathin rods, the following local estimate holds: For any $x\in\Lkexc$, defining $\mathcal{U}(x)=\bigl(\{x_1-\frac{1}{k},x_1,x_1+\frac{1}{k}\}\times\frac{1}{k}\mathcal{L}'\bigr) \cap \Lambda_k'$ we have 
\begin{equation*}
\mathrm{dist}^2(\bar{\nabla}\hat{y}^{(k)}(\hat{x}),\bar{\rm SO}(3))\leq C\sum_{\xi\in k\mathcal{U}(x)}\mathrm{dist}^2(\bar{\nabla}\hat{y}^{(k)}(\xi),\bar{\rm SO}(3)).
\end{equation*}
\end{rem}

\subsection{Elastic energy}
\tc{In the expression for total elastic energy, we group contributions from individual atomic cells} (cf.\ \cite{Conti,BS06}). 

\begin{defn}\label{admCellF}
We say that $W\colon\R^{3\times 8}\goto [0,\infty)$ is a \emph{full cell energy function} if the following assertions hold true:
\begin{enumerate}
\item[(E1)] Frame-indifference: $W(R\vec{y}+(c|\cdots|c))=W(\vec{y})$, $R\in {\rm SO}(3)$, $\vec{y}\in \R^{3\times 8}$, $c\in \R^3$,
\item[(E2)] $W$ attains its minimum (equal to 0) at and only at all rigid deformations, i.e. deformations $\vec{y}=(\hat{y}_1|\cdots|\hat{y}_8)$ with $\hat{y}_i=R\zf^i+c$ for all $i\in\{1,\dots,8\}$ and some $R\in\mathrm{SO}(3)$, $c\in\R^3$,
\item[(E3)] $W$ is everywhere Borel measurable and 
of class $\mathcal{C}^2$ in a neighbourhood of $\bar{\rm SO}(3)$ and the quadratic form associated with $\nabla^2 W(\bar{\Id})$ is positive definite when restricted to $\mathrm{span}\{V_0\cup\R_{\rm skew}^{3\times 3}\bar{\Id}\}^\bot$,
\item[(E4)] $\liminf_{\substack{|\vec{y}|\goto\infty,\\ \vec{y}\in V_0^\bot}}\frac{W(\vec{y})}{|\vec{y}|^2}>0$.
\end{enumerate}
We say that $W\colon\R^{3\times 8}\goto [0,\infty)$ is a \emph{partial cell energy function} if it fulfils (E1) together with
\begin{enumerate}
\item[(E2')] $W$ equals zero for all rigid deformations,
\item[(E3')] $W$ is everywhere Borel measurable and 
of class $\mathcal{C}^2$ in a neighbourhood of $\bar{\rm SO}(3)$.
\end{enumerate}
\end{defn}
Trivially, we see that $W\equiv 0$ is a partial cell energy function.

To model surface energy, let $\mathfrak{T}$ be the power set of $\{1,\ldots,8\}$. We classify the cells centred at $\hat{\Lambda}_k'^{,\mathrm{ext}} = k \Lkexc$ by the set of corners they share with $\hat{\Lambda}_k$, i.e. $\mathfrak{t}_k(\hat{x}) = \big\{ i \in \{1, \ldots, 8\} : \hat{x}+\zf^i \in \hat{\Lambda}_k\bigr\}$ for $\hat{x} \in \hat{\Lambda}_k'^{,\mathrm{ext}}$. (Obviously, $\mathfrak{t}_k(\hat{x}) = \{1, \ldots, 8\}$ iff $\hat{x}\in\hat{\Lambda}_k'$ and $\mathfrak{t}_k(\hat{x})\neq \emptyset$ iff $\hat{x}\in\hat{\Lambda}_k'^{,\rm ext}$ for the specific choice $m=1$.) Also note that on the lateral boundary, i.e.\ for $\hat{x}_1 \notin \{-\frac{1}{2}, kL_k+\frac{1}{2}\}$, we have $i \in \mathfrak{t}_k(\hat{x})$ iff $i+4 \in \mathfrak{t}_k(\hat{x})$ for $i=1,2,3,4$ and so $\mathfrak{t}_k(\hat{x})=\mathfrak{t}_k(\hat{x}'):=\big\{ i \in \{1, \ldots, 8\} : \hat{x}'+(\zf^i)' \in \mathcal{L}_k\bigr\}$. Let $\hat{\Lambda}_k'^{,\mathrm{surf}} = \{\frac{1}{2},\ldots,kL_k-\frac{1}{2}\}\times (\mathcal{L}_k'^{,\rm ext} \setminus \mathcal{L}_k')$ and $\hat{\Lambda}_k'^{,\mathrm{end}} = \{-\frac{1}{2},kL_k+\frac{1}{2}\} \times \mathcal{L}_k'^{,\rm ext}$. Our total elastic interaction energy reads 
\begin{align}\label{eq:EkU}
\begin{split}
E^{(k)}(y^{(k)})
&=\sum_{\hat{x}\in \hat{\Lambda}_k'} W_{\rm cell}\big(\vec{y}^{\,(k)}(\hat{x})\big)+\sum_{\hat{x}\in\hat{\Lambda}_k'^{,\mathrm{surf}}} W_{\rm surf}\big(\mathfrak{t}_k(\hat{x}'),\vec{y}^{\,(k)}(\hat{x})\big)\\
&\quad+\sum_{ \hat{x}\in \hat{\Lambda}_k'^{,\mathrm{end}}} W_{\rm end}\big(\mathfrak{t}_k(\hat{x}),\vec{y}^{\,(k)}(\hat{x})\big),
\end{split}
\end{align}
where $W_{\rm cell}$ is a full cell energy and $W_{\rm surf}(\mathfrak{t},\cdot)$, $W_{\rm end}(\mathfrak{t},\cdot)$ with $\mathfrak{t} \in \mathfrak{T}$ are partial cell energy functions according to Definition \ref{admCellF}. In order to avoid artificial contributions we assume that the values of $W_{\rm surf}(\mathfrak{t},\vec{y})$ and $W_{\rm end}(\mathfrak{t},\vec{y})$, $\vec{y}=(\hat{y}_1|\cdots|\hat{y}_8)$, may depend on $\hat{y}_i$ only if $i \in \mathfrak{t}$. 

We remark that the terms involving $W_{\rm end}$ for cells near the rod's endpoints vanish as $k\goto\infty$ for both ultrathin and thin rods. While for thin rods also the lateral boundary contributions vanish, this is no longer the case for ultrathin rods. Our set-up allows us to model extra-cross-sectional interactions of atoms which lie in different atomic cells but which are, in fact, their mutual neighbours due to the cross section's potentially jagged shape. 

We write $Q_{\rm cell}$ and $Q_{\rm surf}(\mathfrak{t},\cdot)$ for the quadratic forms generated by $\nabla^2 W_{\rm cell}(\bar{\Id})$ and $\nabla^2 W_{\rm surf}(\mathfrak{t},\bar{\Id})$, $\mathfrak{t} \in \mathfrak{T}$, respectively. 

\begin{example}\label{exampU}
To explain the motivation behind $W_{\rm surf}$ and $W_{\rm end}$, let us consider a simple mass-spring model with harmonic springs for a rod with its cross section determined by \tc{$\mathcal{L}_k'\equiv\mathcal{L}'=\{-\frac{1}{2},\frac{1}{2},\frac{3}{2}\}^2\cup\{(\frac{1}{2},-\frac{3}{2}),(\frac{1}{2},\frac{5}{2})\}$ ($m:=1$)}. The aim is to rewrite
$$E^{(k)}(y)=\frac{1}{2}\sum_{\substack{\hat{x}_*,\hat{x}_{**}\in\hat{\Lambda}_k\\ |\hat{x}_*-\hat{x}_{**}|=1}}\frac{K_1}{2}(|\hat{y}(\hat{x}_*)-\hat{y}(\hat{x}_{**})|-1)^2+\frac{1}{2}\sum_{\substack{\hat{x}_*,\hat{x}_{**}\in\hat{\Lambda}_k\\ |\hat{x}_*-\hat{x}_{**}|=\sqrt{2}}}\frac{K_2}{2}(|\hat{y}(\hat{x}_*)-\hat{y}(\hat{x}_{**})|-\sqrt{2})^2$$
using $W_{\rm cell}$, $W_{\rm surf}$, and $W_{\rm end}$ ($K_1>0$ and $K_2>0$ are constant stiffnesses). While in the bulk, we set
\begin{equation*}
W_{\rm cell}(\vec{y})=\frac{1}{8}\sum_{|\zf^i-\zf^j|=1}\frac{K_1}{2}(|\hat{y}_i-\hat{y}_j|-1)^2+\frac{1}{4}\sum_{|\zf^i-\zf^j|=\sqrt{2}}\frac{K_2}{2}(|\hat{y}_i-\hat{y}_j|-\sqrt{2})^2,
\end{equation*}
the functions \tc{$W_{\rm surf}(\mathfrak{t}_k(\frac{5}{2},\frac{3}{2}),\cdot)$, $W_{\rm surf}(\mathfrak{t}_k(-\frac{3}{2},\frac{3}{2}),\cdot)$ etc., and $W_{\rm end}$ in turn include surface terms, e.g.
\begin{gather*}
W_{\rm surf}\Bigl(\mathfrak{t}_k\bigl(\tfrac{5}{2},\tfrac{1}{2}\bigr),\vec{y}\Bigr)=W_{\rm surf}\bigl(\{3,4,7,8\},\vec{y}\bigr)=\sum_{i\in\{3,7\}}\frac{K_1}{8}(|\hat{y}_{i+1}-\hat{y}_i|-1)^2\\
+\sum_{i=3}^4\frac{K_1}{8}(|\hat{y}_{i+4}-\hat{y}_i|-1)^2+\frac{K_2}{4}(|\hat{y}_7-\hat{y}_4|-\sqrt{2})^2+\frac{K_2}{4}(|\hat{y}_8-\hat{y}_3|-\sqrt{2})^2,\\
W_{\rm end}\Bigl(\mathfrak{t}_k\bigl(-\tfrac{1}{2},\tfrac{1}{2},\tfrac{1}{2}\bigr),\vec{y}\Bigr)=W_{\rm end}\bigl(\{1,2,3,4\},\vec{y}\bigr)=\sum_{i=1}^3\frac{K_1}{8}(|\hat{y}_{i+1}-\hat{y}_i|-1)^2\\
+\frac{K_1}{8}(|\hat{y}_4-\hat{y}_1|-1)^2+\sum_{i=1}^2\frac{K_2}{4}(|\hat{y}_{i+2}-\hat{y}_i|-\sqrt{2})^2.
\end{gather*}
The auxiliary square $Q'(\xb_{\rm e}')=\xb_{\rm e}'+(-\frac{1}{2},\frac{1}{2})^2$ centred at $\xb_{\rm e}':=(\frac{3}{2},\frac{5}{2})$ is adjacent to two physically relevant cross-sectional squares, so in particular, the atoms with $\xb'$-coordinates $(2,2)$ and $(1,3)$, belonging to  different `real' atomic squares $Q'(\frac{1}{2},\frac{3}{2})$ and $Q'(\frac{3}{2},\frac{3}{2})$, can still interact -- this interaction should be comprised in $W_{\rm surf}(\mathfrak{t}_k(\xb_{\rm e}'),\cdot)$.} Like this, $E^{(k)}$ is expressible by \eqref{eq:EkU}. After adding a suitable penalty term positive in a neighbourhood of ${\rm O}(3)\bar{\Id}\setminus \bar{\rm SO}(3)$, such $W_{\rm cell}$, $W_{\rm surf}$, and $W_{\rm end}$ fulfil all the assumptions for our results to apply.
\end{example}

\begin{lemma}
\tc{Under the assumptions of Lemma \ref{equivNorms}, let $W$ be a full cell energy function. Then
\begin{align}\label{rig}
  k^3h_k^2\int_{\Q(\bar{\x})}\mathrm{dist}^2(\nabla_k\yk(\xi),\mathrm{SO}(3))\md \xi\leq CW\bigl(\bar{\nabla}\hat{y}^{(k)}(\hat{x})\bigr).
\end{align}}
\end{lemma}
\begin{proof}See \cite[Lemma~3.2]{BS06}; the claim is only restated in our notation.
\end{proof}

\subsection{Compactness of low-energy sequences}

We provide a compactness theorem that complements our $\Gamma$-convergence results in the following sections and is also the first step towards their proofs. It is based on a now well-known result about geometric rigidity from \cite{FrM02} and is essentially contained in \cite{MMh4}. 

We fix a null sequence $(h_k)$ with $\frac{1}{k} \le h_k$ and abbreviate $h'_k=\frac{1}{k}\lfloor k h_k \rfloor$. Set $\Omega_k = (0,L_k) \times \frac{1}{kh_k} S_k$. 

\begin{thm}\label{MMcomp}
Let $(\ybk)_{k=1}^\infty$ be a sequence with $y^{(k)}\colon \Lambda^{\rm ext}_k \to \R^3$ such that their interpolations $\ybk$ constructed in Section~\ref{sec:inter} satisfy the estimate 
\begin{align}\label{eq:dist-SO}
\limsup_{k\goto\infty}\frac{1}{h_k^2}\int_{\Omega_k}\mathrm{dist}^2(\nabla_k\ybk,\mathrm{SO}(3))\md \xb<\infty.
\end{align}
Then, there exist a (not relabelled) subsequence $(\ybk)$ and a sequence of piecewise constant mappings $R^{(k)}\colon \R\goto {\rm SO}(3)$ whose discontinuity set is contained in $\{h'_k,2h'_k,\ldots, (\lfloor L_k/h'_k \rfloor-1) h'_k\}$ such that
\begin{align}
R^{(k)}\goto R\text{ in }L^2([0,L];\R^{3\times 3}),\label{eq:RkRU}
\end{align}
where $R\in\mathrm{SO}(3)$ a.e.\ and $R(x_1)=\Bigl(\frac{\partial \yb}{\partial \xb_1}(x)\,\big|\,d_2(x)\,\big|\,d_3(x)\Bigr)$ for $\y\in H^2(\O;\R^3)$, $d_2,d_3\in H^1(\O;\R^3)$ that are independent of $\x_2$ and $\x_3$. Moreover, we have
\begin{align}
\int_{\Omega^{\rm ext}_k}|\nabla_k\ybk-R^{(k)}|^2\md\xb \leq Ch_k^2.\label{eq:FrMU}
\end{align}
and
\begin{align}
|R^{(k)}(ih'_k+\tfrac{3}{2}h'_k)-R^{(k)}(ih'_k+\tfrac{1}{2}h'_k)|^2
&\leq \frac{C}{h_k}\lVert\mathrm{dist}(\nabla_k\ybk,{\rm SO}(3))\rVert_{L^2((u_i^{(k)},v_i^{(k)})\times S^{\rm ext}_k)}^2,\label{eq:ptCurv}
\end{align}
where $u_i^{(k)}=ih_k'$, $v_i^{(k)}=u_i^{(k)}+2h_k$ for $i=1,\dots,\lfloor L_k/h'_k \rfloor-3$, $u_0^{(k)}=-\frac{1}{k}$, $v_0^{(k)}=-\frac{1}{k}+3h_k$ and $u_i^{(k)}=\min\{(\lfloor L_k/h'_k \rfloor -2)h_k',L_k+\frac{1}{k}-3h_k\}$, $v_i^{(k)}=L_k+\frac{1}{k}$ for $i=\lfloor L_k/h'_k \rfloor-2$.
\end{thm}

Note that in the ultrathin case one has $h_k = 1/k$ and so $h'_k=\frac{1}{k}$, $(\lfloor L_k/h'_k \rfloor-1) h'_k = L_k-\frac1k$.  

\begin{proof}
By Lemma~\ref{lemma:ext-general}, property \eqref{eq:dist-SO} is equivalent to 
\begin{align*}
\limsup_{k\goto\infty}\frac{1}{h_k^2}\int_{\Omega^{\rm ext}_k}\mathrm{dist}^2(\nabla_k\ybk,\mathrm{SO}(3))\md \xb<\infty,
\end{align*}
hence also to the same inequality with $\Omega_k^{\rm ext}$ replaced by $(-\frac{1}{k},L_k+\frac{1}{k}) \times S$ or $(0,L)\times S$. 
Except for the specific choice of the discontinuity set, these statements are thus proven in \cite{MMh4} by applying the geometric rigidity theorem of \cite{FrM02} to sets of the form $(a, a + bh_k) \times h_kS$.\footnote{Because of lattice squares that only share one corner, in the ultrathin case $S_k$ might not be Lipschitz, but as a finite union of squares the domain is still very regular so that all necessary claims hold. Thus we can choose $S=S_k$ for ultrathin rods.} If we do this here for $b = 1$ and the special choices $a=ih_k'$, $i=1,\ldots,\lfloor L_k/h'_k \rfloor-2$ as well as $b=3$ and $a\in\{-\frac{1}{k},L_k+\frac{1}{k}-3h_k\}$, we see that $R^{(k)}$ can be arranged to jump only in $\{h'_k,2h'_k,\ldots, (\lfloor L_k/h'_k \rfloor-1) h'_k\}$.
\end{proof}

We remark that, for a suitable choice of translation vectors $c_k$ (which does not change the energy),  $\ybk-c_k \goto \yb$ in $H^1(\O;\R^3)$.

\section{Resulting theory for ultrathin rods}\label{sec:U}

We now specialize to ultrathin rods for which the cross sectional lattice $\mathcal{L}_k = \mathcal{L}$ is assumed to be fixed. We set $h_k = 1/k$ and fix $m=1$. Since also $\mathcal{L}_k^{\rm ext}$, $\mathcal{L}'^{,\rm ext}_k$, $S_k$ (which equals $S$ without loss of generality) and $S^{\rm ext}_k$ are independent of $k$, we drop the subscript.

\subsection{\tc{Difference operators}}\label{sec:diff-ops}
In addition to $\bar{\nabla}$, we define several other difference operators, applicable to any $f\colon [-\frac{1}{k},L_k+\frac{1}{k}]\times\mathcal{L}^{\rm ext}\goto\R^\ell$, $\ell\in\N$. If $x\in \overline{\Omega^{\rm ext}_k}$, we denote by $\bar{x}$ an element of $\Lkexcb$ that is closest to $x$. For $x\in\overline{\Omega^{\rm ext}_k}$ we set 
\begin{align*}
\dk2d f(\xb)
&=k\biggl[f(\xb_1,(\bar{\xb}+\zfb^i)')-\frac{1}{4}\sum_{j=1}^4 f(\xb_1,(\bar{\xb}+\zfb^j)')\biggr]_{i=1}^4,\\
\d2d f(\xb)
&=\biggl[f(\xb_1,(\bar{\xb}+\zf^i)')-\frac{1}{4}\sum_{j=1}^4 f(\xb_1,(\bar{\xb}+\zf^j)')\biggr]_{i=1}^4,\\
\dgradk f(\xb)
&=k\biggl[f(\bar{\xb}+\zfb^i)-\frac{1}{8}\sum_{j=1}^8 f(\bar{\xb}+\zfb^j)\biggr]_{i=1}^8\\
&=\bigl(\dleft f(\xb)|\dright f(\xb)\bigr),\qquad
\dleft f(\xb),\;\dright f(\xb)\in\R^{3\times 4},\\
\Delta_1 f(\xb)&=k\biggl[\frac{1}{4}\sum_{j=5}^8 f(\bar{\xb}+\zfb^j)-\frac{1}{4}\sum_{j=1}^4 f(\bar{\xb}+\zfb^j)\biggr],
\end{align*}
whose interpretations are `2D-differences in the $\xb_2\xb_3$-plane' (divided by $1/k$ or not), `3D-differences' and `averaged difference in the $\xb_1$-direction', respectively. Note that the functions $\dk2d f(\xb_1,\cdot)$ and $\d2d f(\xb_1,\cdot)$ are piecewise constant on lattice squares of the form $\xb'+(-\frac{1}{2},\frac{1}{2})^2$, where $\xb'\in\mathcal{L}'$, and $\d2d f(\xb)$ is independent of $k$. The functions $\dgradk f$ and $\Delta_1 f$ are piecewise constant on lattice blocks that are centred in points of $\Lkexcb$.

Set $\yb_i^{(k)}=\ybk(\bar{\xb}+\zfb^i)$, $i=1,2,\dots,8$, then property \eqref{eq:surfVolMean} yields
\begin{align}\label{eq:avgDiff1}
\begin{split}
\Delta_1\ybk(\xb)
&=k\bigg(\frac{\yb_5^{(k)}+\yb_6^{(k)}+\yb_7^{(k)}+\yb_8^{(k)}}{4}-\frac{\yb_1^{(k)}+\yb_2^{(k)}+\yb_3^{(k)}+\yb_4^{(k)}}{4}\bigg)\\
&=k\dashint_{\bar{\xb}'+(-\frac{1}{2},\frac{1}{2})^2}\ybk(\bar{\xb}_1+\tfrac{1}{2k},\xi')-\ybk(\bar{\xb}_1-\tfrac{1}{2k},\xi')\md\xi'\\
&=k\dashint_{\bar{\xb}'+(-\frac{1}{2},\frac{1}{2})^2}\int_{\bar{\xb}_1-\frac{1}{2k}}^{\bar{\xb}_1+\frac{1}{2k}}\frac{\partial\ybk}{\partial\xb_1}(\xi_1,\xi')\md\xi_1\md\xi'=\dashint_{\Qb(\xb)}\frac{\partial\ybk}{\partial\xb_1}\md\xi.
\end{split}
\end{align}

Direct computation shows:
\begin{align*}
\dleft f(\xb)
&=\dk2d f(\bar{\xb}_1-\tfrac{1}{2k}, \xb')-\frac{1}{2}\Delta_1 f(\xb)(1,1,1,1),\\
\dright f(\xb)
&=\tc{\dk2d f(\bar{\xb}_1+\tfrac{1}{2k}, \xb')+\frac{1}{2}\Delta_1 f(\xb)(1,1,1,1)}
\end{align*}
and so, with all columns grouped together,
\begin{equation}\label{eq:splitGrad}
\tc{\dgradk f(\xb)
=\bigl(\dk2d f(\bar{\xb}_1-\tfrac{1}{2k}, \xb')|\dk2d f(\bar{\xb}_1+\tfrac{1}{2k}, \xb')\bigr)+\frac{1}{2}\Delta_1 f(\xb)(-1,-1,-1,-1,1,1,1,1).}
\end{equation}

\subsection{Gamma-convergence}\label{sec:GammaU}

Recall that $\O=(0,L)\times S$. In order to specify an appropriate limit space we first note that in view of Theorem~\ref{MMcomp} and \eqref{rig} it suffices to consider limiting configurations $\yb\in H^1(\Omega;\R^3)$ and $d_2,d_3 \in L^2(\Omega;\R^3)$ that do not depend on $(x_2,x_3)$. We will then simply write $\yb\in H^1((0,L);\R^3)$ and $d_2,d_3 \in L^2((0,L);\R^3)$. The following observation shows that the convergence in $L^2(\O,\R^3)$ to such $\yb$ and $d_2,d_3$ is naturally described in terms of asymptotic atomic positions and independent of our interpolation scheme, cf.\ also Remark~\ref{rmk:gen-lim}. 

For a sequence $(y^{(k)})_{k=1}^\infty$ of (extended) lattice deformations and $\yb\in H^1((0,L);\R^3)$ the convergence $\ybk \to \yb$ in $L^2(\Omega;\R^3)$ is equivalent to 
$$ \ybk(\cdot,x') \to \yb 
   \text{ in } L^2((0,L);\R^3) 
   \text{ for every } x' \in \tc{\mathcal{L}}.$$
We note here that for $x' \in \tc{\mathcal{L}}$ the map $\yk(\cdot,x')$ is nothing but the piecewise affine interpolation of $\yk(\cdot,x')$ on $\{-\frac{1}{k}, 0, \ldots, L_k+\frac{1}{k}\}$. If moreover $d_2, d_3 \in L^2((0,L);\R^3)$, then \tc{$\nabla_k \ybk \stackrel{L^2}{\to} R=(\frac{\partial \yb}{\partial \xb_1}\,|\,d_2\,|\,d_3)$} is equivalent to 
$$ \tc{\bar{\nabla}_k\ybk\goto R\,\bar{\Id}\text{ in }L^2(\O;\R^{3\times 8})} $$
\tc{by Lemma \ref{equivNorms} (recall that $\bar{\nabla}_k\ybk$ is a function in $L^2(\O;\R^{3\times 8})$ constant on each cell $\bar{x}+(-\frac{1}{2k},\frac{1}{2k})\times(-\frac{1}{2},\frac{1}{2})^2$, $\bar{x} \in \Lkexcb$).}

\begin{thm}\label{GammaU}
If $k\goto\infty$, the functionals $kE^{(k)}$ $\Gamma$-converge to the functional $E_{\rm ult}$ defined below, in the following sense:
\begin{enumerate}
\item[(i)] \upshape{(liminf inequality)} \itshape
Let $(y^{(k)})_{k=1}^\infty$  be a sequence of (extended) lattice deformations such that their piecewise affine interpolations $(\ybk)_{k=1}^\infty$, defined in Section \ref{sec:not}, converge to $\yb\in H^1((0,L);\R^3)$ in $L^2(\O;\R^3)$. Let us also assume that $k\pl_{\xb_s}\ybk\goto d_s\in L^2((0,L);\R^3)$ in $L^2(\O;\R^3)$, $s=2,3$. Then
$$E_{\rm ult}(\yb,d_2,d_3)\leq \liminf_{k\goto\infty} kE^{(k)}(y^{(k)}).$$
\item[(ii)] \upshape{(existence of a recovery sequence)} \itshape For every $\yb\in H^1((0,L);\R^3)$, $d_2,d_3\in L^2((0,L);\R^3)$ there is a sequence of (extended) lattice deformations $(y^{(k)})_{k=1}^\infty$ such that their interpolations $(\ybk)$, defined in Section \ref{sec:not}, satisfy $\ybk\goto\yb$ in $L^2(\O;\R^3)$, $k\frac{\pl\ybk}{\pl\xb_s}\goto d_s$ in $L^2(\O;\R^3)$ for $s=2,3$, and
$$\lim_{k\goto\infty} kE^{(k)}(y^{(k)})=E_{\rm ult}(\yb,d_2,d_3). $$
\end{enumerate}
The limit energy functional is given by
\begin{equation*}
E_{\rm ult}(\yb,d_2,d_3)=\begin{cases}
\tc{\frac{1}{2}}\int_0^L Q_{\rm cell}^{\rm rel}(\T{R}\pl_{\xb_1}R)\md\xb_1 & \text{if } (\yb,d_2,d_3)\in\calA,\\
+\infty & \text{otherwise},
\end{cases}
\end{equation*}
where $R:=(\pl_{\xb_1}\yb|d_2|d_3)$ and the class of admissible deformations is
\begin{multline*}
\calA:=\bigl\{(\yb,d_2,d_3)\in H^2(\O;\R^3)\times H^1(\O;\R^3)\times H^1(\O;\R^3);\\
\yb,d_2,d_3\textit{ do not depend on }\xb_2,\xb_3,\quad \bigl(\tfrac{\pl\yb}{\pl\xb_1}\,|\,d_2\,|\, d_3\bigr)\in{\rm SO}(3)\text{ a.e.\ in }(0,L)\bigr\}.
\end{multline*}
The relaxed quadratic form $Q_{\rm cell}^{\rm rel}\colon\R^{3\times 3}_{\rm skew}\goto[0,+\infty)$ is defined as
\begin{multline}\label{eq:minProbU}
Q_{\rm cell}^{\rm rel}(A):=\min_{\substack{\a\colon\mathcal{L}^{\rm ext}\goto\R^3\\g\in\R^3}}\sum_{\xb'\in\mathcal{L}'^{,\rm ext}} Q_{\rm tot}\biggl(\xb',\frac{1}{2}
\bigl(A\T{(0,x')}+g\bigr)(-1,-1,-1,-1,1,1,1,1)\\
+\frac{1}{4}A{\sst\begin{pmatrix}[0.7]
\sst 0&
\sst 0&
\sst 0&
\sst 0&
\sst 0&
\sst 0&
\sst 0&
\sst 0\\
\sst 1&
\sst 1&
\sst -1&
\sst -1&
\sst -1&
\sst -1&
\sst 1&
\sst 1\\
\sst 1&
\sst -1&
\sst -1&
\sst 1&
\sst -1&
\sst 1&
\sst 1&
\sst -1
\end{pmatrix}}
+\bigr(\d2d\a|\d2d\a\bigr)\biggr)
\end{multline}
with $Q_{\rm tot}(\xb',\cdot)=Q_{\rm cell}$ if $x'\in \mathcal{L}'$ and $Q_{\rm tot}(\xb',\cdot)=Q_{\rm surf}(\mathfrak{t}({\xb'}),\cdot)$ if $x'\in \mathcal{L}'^{,\rm ext} \setminus \mathcal{L}'$.
\end{thm}
\begin{rem}
In comparison with the rod theory in \cite{MMh4}, the functional $E_{\rm ult}$ takes into account the fewer degrees of freedom of the cross section leading to the discrete minimization in \eqref{eq:minProbU} and it also features an \textit{ultrathin correction} term $\mathfrak{C}$, explicitly given in \eqref{eq:thinCorr} below, which captures effects in our very thin atomic structures that could not be described by a Cauchy--Born continuum approximation.
\end{rem}
\begin{rem}

Let us comment on the existence of a minimizer in \eqref{eq:minProbU}. Fix $A\in\R^{3\times 3}_{\rm skew}$ and let
$$C_A(\xb')=\frac{1}{2}A\T{(0,x')}(-1,-1,-1,-1,1,1,1,1)+\frac{1}{4}A{\sst\begin{pmatrix}[0.7]
\sst 0&
\sst 0&
\sst 0&
\sst 0&
\sst 0&
\sst 0&
\sst 0&
\sst 0\\
\sst 1&
\sst 1&
\sst -1&
\sst -1&
\sst -1&
\sst -1&
\sst 1&
\sst 1\\
\sst 1&
\sst -1&
\sst -1&
\sst 1&
\sst -1&
\sst 1&
\sst 1&
\sst -1
\end{pmatrix}}.$$
The mapping $J\colon (\R^3)^{\mathcal{L}^{\rm ext}}\times\R^3\goto\R$ given by
\begin{equation*}
J(\a,g)=\sum_{\xb'\in\mathcal{L}'^{,\rm ext}}Q_{\rm tot}\bigl(\xb',C_A(\xb')+\bigl(\d2d\a(\xb')|\d2d\a(\xb')\bigr)+\tfrac{1}{2}(-g|\cdots|-g|g|\cdots|g)\bigr)
\end{equation*}
is, in fact, a real-valued function of $3\cdot\sharp\mathcal{L}^{\rm ext}+3$ variables. Since $Q_{\rm tot}(\xb',\cdot)$ is positive semidefinite on $\R^{3\times 8}$ for any $\xb'\in\mathcal{L}'^{,\rm ext}$, the function $J$ is a positive semidefinite quadratic form. It thus attains a minimum and a minimizer $(\a,g)$ of $J$ can be chosen in linear dependence on $A$, so $Q_{\rm cell}^{\rm rel}$ is a quadratic form as well. Besides, since the components of $A:=\T{R}\pl_{\xb_1}R$ are $L^2$ in $\xb_1$, we obtain $(\a,g)\in L^2([0,L];(\R^3)^{\mathcal{L}^{\rm ext}}\times\R^3)$.
\end{rem}
\begin{rem}\label{rmk:gen-lim} One could also consider limiting configurations with an explicit dependence on $x'$. Due to the discrete nature of $\mathcal{L}$, however, only a subspace of $H^1(\O;\R^3)$ can be realized as limits of interpolated deformations $\ybk$. \tc{That is, $\ybk$ can converge to $\yb$} in $L^2(\Omega;\R^3)$  if and only if $\yb$ is piecewise affine in $x'$, more precisely, if for a.\ e.\ $x_1\in (0,L)$ and $x' \in \mathcal{L}'$ one has $\tilde{y}(x_1,x') = \frac{1}{4} \sum_{i=1}^4 \tilde{y}(x_1,x'+(\zf^i)')$ and $\tilde{y}(x_1, \cdot)$ is affine on $\mathrm{conv}\{x',x'+(\zf^i)',x'+(\zf^j)'\}$ if $i,j\in\{1,2,3,4\}$, $|i-j|=1$. Similarly, limiting directors $(d_2,d_3)$ are restricted to be gradients with respect to $x'$ of such functions. By Theorem~\ref{MMcomp}, \eqref{eq:EkU} and \eqref{rig} one still has $\Gamma$-convergence with such a class of limiting configurations if $E_{\rm ult}$ is extended by the value $+\infty$  outside of $H^1((0,L);\R^3)\times L^2((0,L);\R^3) \times L^2((0,L);\R^3)$.
\end{rem}
\begin{rem}
\tc{Standard arguments show that forcing terms of the form $-k^{-3}h_k^{-2}\sum_{x\in\Lambda_k}f(x_1)\cdot y^{(k)}(x)$, $f\in L^2((0,L);\R^3)$, could be added to $k^{-3}h_k^{-4}E^{(k)}$ and $\Gamma$-convergence as well as compactness claims would still hold (see e.g. \cite[Corollary~3.4]{BS07} or \cite{BrS19} for more details).}
\end{rem}

\subsection{Proof of the lower bound}
In this section, we prove Theorem \ref{GammaU}(i). We may assume that $kE^{(k)}(y^{(k)}) \le C$ and so \eqref{eq:dist-SO} holds true by \eqref{eq:EkU} and \eqref{rig}.
 We set $\Os=(0,L)\times S^{\rm ext}$. Let $R^{(k)}$ be as in  Theorem \ref{MMcomp}. By \eqref{eq:FrMU} and in analogy with \cite{MMh4}, for
\begin{equation*}
\Gk(\xb):=\frac{\T{(R^{(k)})}(\xb_1)\nabla_k\ybk(\xb)-\Id}{1/k},\quad \xb\in\Omega^{\rm ext}_k,
\end{equation*}
we have $\Gk\weakto G\in L^2(\Os;\R^{3\times 3})$ in $L^2(\Os;\R^{3\times 3})$, up to a subsequence. In our discrete setting we instead need to study
\begin{equation*}
\bGk(\xb):=\frac{\T{(R^{(k)})}(\xb_1)\dgradk \ybk-\bar{\Id}}{1/k},\quad \xb\in\Omega^{\rm ext}_k.
\end{equation*}
The $L^2$-boundedness of $\{\Gk\}$ implies the boundedness of $\{\bGk\}$ in $L^2(\Os;\R^{3\times 8})$ by \eqref{eq:norms}. Hence $\bGk\weakto\bar{G}$ for a subsequence, which we do not relabel. 
 We state a proposition about the structure of $\bar{G}$. 

\begin{prop}\label{Gbar}
$\bGk\weakto\bar{G}$ in $L^2(\Os;\R^{3\times 8})$ for 
\begin{align*}
\bar{G}(\xb)
&=\frac{1}{2}\Bigl[G_1(\xb_1)+\T{R}(\xb_1)\frac{\pl R}{\pl\xb_1}(\xb_1)\T{(0,\bar{x}')}\Bigr](-1,-1,-1,-1,1,1,1,1)\\
&\quad +\mathfrak{C}(\xb_1)+\bigl(\d2d\a(\xb)|\d2d\a(\xb)\bigr),
\end{align*}
where $G_1 \in L^2((0,L);\R^3)$, $\alpha\in L^2((0,L)\times\mathcal{L}^{\rm ext};\R^3) \cong L^2((0,L);(\R^3)^{\mathcal{L}^{\rm ext}})$ and $\mathfrak{C}$ is explicitly given by 
\begin{align}
\mathfrak{C}
&=\frac{1}{4}\begin{pmatrix}[0.8]
\sst -\kappa_2-\kappa_3&
\sst \kappa_3-\kappa_2&
\sst \kappa_2+\kappa_3&
\sst \kappa_2-\kappa_3&
\sst \kappa_2+\kappa_3&
\sst \kappa_2-\kappa_3&
\sst -\kappa_2-\kappa_3&
\sst \kappa_3-\kappa_2\\
\sst -\tau&
\sst \tau&
\sst \tau&
\sst -\tau&
\sst \tau&
\sst -\tau&
\sst -\tau&
\sst \tau\\
\sst \tau&
\sst \tau&
\sst -\tau&
\sst -\tau&
\sst -\tau&
\sst -\tau&
\sst \tau&
\sst \tau
\end{pmatrix}\label{eq:thinCorr}
\end{align}
with $\kappa_2(\xb_1)=\frac{\pl^2\yb}{\pl\xb_1^2}\cdot d_2$, $\kappa_3(\xb_1)=\frac{\pl^2\yb}{\pl\xb_1^2}\cdot d_3$, $\tau(\xb_1)=\frac{\pl d_2}{\pl \xb_1}\cdot d_3$, and $R$ from Theorem \ref{GammaU}. 
\end{prop}
\begin{proof}
Formula \eqref{eq:splitGrad} enables us to find the longitudinal and transversal contributions separately.

\noindent 1. \textit{Longitudinal contributions.} We consider the piecewise constant function 
$$ \bar{G}_{\rm long}^{(k)}
  :=\frac{k}{2}\Bigl[\T{(R^{(k)})}\Delta_1 \ybk(-1,-1,-1,-1,1,1,1,1)-e_1\T{e_1}\bar{\Id}\Bigr] $$ 
and observe that for each $x \in \Omega^{\rm ext}_k$ with $\Qb(x):=\Qb(\bar{\xb})=\bar{\xb}+(-\frac{1}{2k},\frac{1}{2k})\times (-\frac{1}{2},\frac{1}{2})^2$ property \eqref{eq:avgDiff1} of the piecewise affine interpolation $\ybk$ yields 
\begin{align*}
\bar{G}_{\rm long}^{(k)}
&=\frac{k}{2}\Bigl(\T{(R^{(k)})}\dashint_{\Qb}\frac{\pl\ybk}{\pl \xb_1}\md\xi-e_1\Bigr)(-1,-1,-1,-1,1,1,1,1)\\
&=\frac{1}{2}\dashint_{\Qb}\Gk e_1\md\xi(-1,-1,-1,-1,1,1,1,1).
\end{align*}
This converges weakly to
$$\frac{1}{2}\dashint_{\Qb'(\xb')}G(\xb_1,\xi')e_1\md\xi'(-1,-1,-1,-1,1,1,1,1),$$
where $\Qb'(\xb')=(\bar{\xb}_2-\frac{1}{2},\bar{\xb}_2+\frac{1}{2})\times(\bar{\xb}_3-\frac{1}{2},\bar{\xb}_3+\frac{1}{2})$, since for any $\varphi\in C^{\infty}_c(\Os)$ 
\begin{align*} 
\int_{\Os}\dashint_{\Q(\bar{x})}\Gkn(\xi)\md\xi \varphi(x) \md x 
&= \int_{\Os}\Gkn(\xi) \dashint_{\Q(\bar{\xi})} \varphi(x) \md x  \md \xi \\ 
&\goto \int_{\Os} G(\xi) \int_{\Qb'(\xi')} \varphi(\xi_1,x') \md x' \md \xi\\ 
&= \int_{\Os} \int_{\Qb'(\xb')} G(x_1,\xi') \md \xi' \varphi(x) \md x.
\end{align*}
(A similar property is also used in \cite[Proposition~4.6]{BrS19}.)
In view of \cite[equation~(3.10)]{MMh4}, the first column of $G$ reads
$$Ge_1=G_1(\xb_1)+\T{R}(\xb_1)\frac{\pl R}{\pl \xb_1}(\xb_1)\T{(0, \xb')}$$
for some $G_1 \in L^2((0,L);\R^3)$ and hence
$$\dashint_{\Qb'(\xb')}G(\xb_1,\xi')e_1\md\xi'=G_1(\xb_1)+\T{R}(\xb_1)\frac{\pl R}{\pl \xb_1}(\xb_1)\T{(0, \bar{\xb})}.$$
It follows that in $L^2(\Os;\R^{3\times 8})$,
\begin{equation*}
\bGk_{\rm long}\weakto\frac{1}{2}\bigl(G_1(\xb_1)+\T{R}(\xb_1)\frac{\pl R}{\pl \x_1}(\xb_1)\T{(0, \bar{\xb})}\bigr)(-1,-1,-1,-1,1,1,1,1). 
\end{equation*}

\noindent 2. \textit{Transversal contributions.} Here the left $3\times 4$ submatrix of some $\bar{A}\in\R^{3\times 8}$ is referred to as the \textit{left part} of $\bar{A}$, whereas the other $3\times 4$ submatrix as the \textit{right part} of $\bar{A}$.

First let us look at the left part
\begin{equation*}
\bGk_{\rm left}(\xb):=k\Bigl[\T{(R^{(k)})}(\xb_1)\dk2d\ybk(\tc{\bar{\xb}_1-\tfrac{1}{2k},\bar{\xb}'})-\begin{pmatrix}[0.7]
\sst 0&
\sst 0&
\sst 0&
\sst 0\\
\sst (\zf^1)'&
\sst (\zf^2)'&
\sst (\zf^3)'&
\sst (\zf^4)'
\end{pmatrix}\Bigr].
\end{equation*}
We define the auxiliary function
\begin{equation*}
\a_{\rm left}^{(k)}(\xb):=k\big[k\T{R^{(k)}(\xb_1)}\ybk(\bar{\xb}_1-\tfrac{1}{2k},\xb')-\T{(0,\xb')}\big],\quad \xb\in \overline{\Omega^{\rm ext}_k},
\end{equation*}
whose average over the cross-sectional lattice is
$$\a_{\mathrm{left},0}^{(k)}(\xb_1):=\frac{1}{\sharp\mathcal{L}^{\rm ext}}\sum_{\xb'\in\mathcal{L}^{\rm ext}}\a_{\rm left}^{(k)}(\xb_1,\xb')$$
and its two-dimensional discrete gradient is equal to $\bGk_{\rm left}$, since
$$\d2d\a_{\rm left}^{(k)}(\xb)=k\Bigl[\T{(R^{(k)}(\xb_1)}\dk2d\ybk(\bar{\xb}_1-\tfrac{1}{2k},\bar{\xb}')-\frac{1}{2}\begin{pmatrix}[0.7]
\sst 0&
\sst 0&
\sst 0&
\sst 0\\
\sst -1&
\sst -1&
\sst 1&
\sst 1\\
\sst -1&
\sst 1&
\sst 1&
\sst -1
\end{pmatrix}\Bigr].$$

Since $S^{\rm ext}$ is a polygonal domain, \tc{setting $\nabla'f=(\pl_{\xb_2}f\,|\,\pl_{\xb_3} f)$ and bounding the $\max_{S^{\rm ext}}$ with a successive maximization over $\mathcal{L}'^{,\rm ext}$ and over the interpolation tetrahedra,} we have
\begin{align*}
  |\a_{\rm left}^{(k)}(\xb)-\a^{(k)}_{\mathrm{left},0}(\xb_1)|^2 
  &\le C \max_{\zeta'\in S^{\rm ext}} |\nabla'\a_{\rm left}^{(k)}(\xb_1,\zeta')|^2  \\
  &=C \max_{\zeta'\in S^{\rm ext}} k^2\biggl|\T{R^{(k)}(\xb_1)}\bigl[\tc{k\nabla'}\ybk(\bar{\xb}_1-\tfrac{1}{2k},\zeta')\bigr]-
\begin{pmatrix}[0.7]
\sst 0&
\sst 0\\
\sst 1&
\sst 0\\
\sst 0&
\sst 1
\end{pmatrix}\biggr|^2 \\ 
  &\le 24\max_{\zeta'\in\mathcal{L}'^{,\rm ext}} C \dashint_{\tilde{Q}(\bar{x}_1,\zeta')} k^2\bigl|\T{R^{(k)}(\xi_1)}\nabla_k\ybk(\xi)-
\Id\bigr|^2\md \xi \\
  &\le C \dashint_{\bar{x}_1-\frac{1}{2k}}^{\bar{x}_1+\frac{1}{2k}} \int_{S^{\rm ext}} \bigl|G^{(k)}(\xi)\bigr|^2\md \xi. 
\end{align*}
Integrating over $\Omega^{\rm ext}_k$ shows that $\a_{\rm left}^{(k)}-\a^{(k)}_{\mathrm{left},0}$ and $\partial_{\xb_s}(\a_{\rm left}^{(k)}-\a^{(k)}_{\mathrm{left},0})=\partial_{\xb_s}\a_{\rm left}^{(k)}$, $s=2,3$, are bounded in $L^2(\Omega^{\rm ext}_k;\R^3)$. 
We thus find $\a_{\rm left}\in L^2(\Os;\R^3)$ with $\nabla'\a_{\rm left}\in L^2(\Os;\R^{3\times 2})$ such that, passing to a subsequence,
$$\a_{\rm left}^{(k)}-\a_{\mathrm{left},0}^{(k)}\weakto \a_{\rm left} 
\quad\text{and}\quad 
\partial_{\xb_s} \big( \a_{\rm left}^{(k)}-\a_{\mathrm{left},0}^{(k)} \big) \weakto \partial_{\xb_s} \a_{\rm left},~ s=2,3,$$ 
in $L^2(\Os;\R^3)$. 
In particular, for any $\bar{x}' \in \mathcal{L}'^{,\rm ext}$ and $i,j \in \{1,2,3,4\}$ with \tc{$|\zf^i-\zf^j|=1$} considering the triangle $T=\mathrm{conv}\{\bar{x}',\bar{x}'+(\zf^i)',\bar{x}'+(\zf^j)'\} \subset S^{\rm ext}$ we still have 
$$ \nabla' \big( \a_{\rm left}^{(k)}-\a_{\mathrm{left},0}^{(k)} \big) \weakto \nabla'\a_{\rm left} 
\quad\text{in } 
L^2((0,L) \times T;\R^{3\times 2}). $$
Our piecewise affine interpolation scheme and the definition of $\a_{\rm left}^{(k)}$ guarantee that $\nabla'  \a_{\rm left}^{(k)}(x)$ is independent of $x'$ and piecewise constant in $x_1$ for $x\in (0,L) \times T$. Therefore \tc{ $\nabla'\a_{\rm left}(x)$ does not depend on $x' \in T$ either} and we may conclude that 
$$ \big[\d2d \big( \a_{\rm left}^{(k)}-\a_{\mathrm{left},0}^{(k)} \big)\big]_{\bullet i}
  = \nabla' \big( \a_{\rm left}^{(k)}-\a_{\mathrm{left},0}^{(k)} \big) (\zf^i)'
  \weakto \nabla'\a_{\rm left}  (\zf^i)' 
  = \big[\d2d \a_{\rm left} \big]_{\bullet i}$$
in $L^2((0,L) \times T;\R^3)$. As 
both $\d2d ( \a_{\rm left}^{(k)}-\a_{\mathrm{left},0}^{(k)} )$ and $\d2d \a_{\rm left}$ are in fact independent of $x' \in \Qb'=\bar{x}'+(-\frac{1}{2},\frac{1}{2})^2$, we even have $[\d2d ( \a_{\rm left}^{(k)}-\a_{\mathrm{left},0}^{(k)} )]_{\bullet i} \weakto [\d2d \a_{\rm left} \big]_{\bullet i}$ in $L^2((0,L) \times \Qb';\R^3)$ and so, since $\bar{x}'$ and $i$ were arbitrary, 
$$\d2d(\a_{\rm left}^{(k)}-\a_{\mathrm{left},0}^{(k)})=\bar{G}_{\rm left}^{(k)}\weakto \d2d\a_{\rm left}\text{ in }L^2(\Os;\R^{3\times 4})$$
and the restriction of $\a_{\rm left}$ to $(0,L)\times\mathcal{L}^{\rm ext}$ is well defined. Similarly we find $\a_{\rm right}\in L^2(\Os;\R^3)$, the weak limit of $\a_{\rm right}^{(k)}-\a_{\mathrm{right},0}^{(k)}$, so that
\begin{equation*}
\bGk_{\rm right}=k\Bigl[\T{(R^{(k)})}\dk2d\ybk(\,\tc{\bar{\cdot}+\tfrac{1}{2k}e_1})-\begin{pmatrix}[0.7]
\sst 0&
\sst 0&
\sst 0&
\sst 0\\
\sst (\zf^5)'&
\sst (\zf^6)'&
\sst (\zf^7)'&
\sst (\zf^8)'
\end{pmatrix}\Bigr]\weakto \d2d\a_{\rm right}.
\end{equation*}

It would be nice to express $\a_{\rm right}$ in terms of $\a_{\rm left}$ and $R$. We see that
$$\a_{\rm right}^{(k)}(\xb-\tfrac{1}{k}e_1)=k\bigl[k\T{(R^{(k)}(\xb_1-\tfrac{1}{k})}\ybk(\tc{\bar{\xb}_1-\tfrac{1}{2k}},x')-\T{(0,\xb')}\bigr]$$
and so
$$\a_{\rm right}^{(k)}(\xb-\tfrac{1}{k}e_1)-\a_{\rm left}^{(k)}(\xb)=k^2\T{\bigl[(R^{(k)})(\xb_1-\tfrac{1}{k})-(R^{(k)})(\xb_1)]}\ybk(\bar{\xb}_1-\tfrac{1}{2k},x').$$
For the discrete gradient of the above expression, we have
\begin{align}\label{eq:dgalar}
\begin{split}
\MoveEqLeft\d2d\bigl(\a_{\rm right}^{(k)}(\xb-\tfrac{1}{k}e_1)-\a_{\rm left}^{(k)}(\xb)\bigr)\\
&=\frac{\T{(R^{(k)})}(\xb_1-\frac{1}{k})-\T{(R^{(k)})}(\xb_1)}{1/k}\dk2d\ybk(\bar{\xb}_1-\tfrac{1}{2k},x').
\end{split}
\end{align}
From \eqref{eq:ptCurv} and \eqref{eq:dist-SO}, we see that $\bigl(k(R^{(k)}(\cdot-\frac{1}{k})-R^{(k)})\bigr)_{k\in\N}$ is bounded in $L^2$ and therefore has a subsequence weakly converging to, say, $F$. 
Moreover, convergence \eqref{eq:RkRU} gives $F=-\frac{\pl R}{\pl \xb_1}$, thus
\begin{equation}\label{eq:wConvdRdx}
k\bigl(R^{(k)}(\cdot-\tfrac{1}{k})-R^{(k)}\bigr)\weakto-\frac{\pl R}{\pl \xb_1}\text{ in }L^2(\Os;\R^{3\times 3}).
\end{equation}
Now we notice that
\begin{equation}\label{eq:d2dykConv}
\dk2d\ybk(\bar{\xb}_1-\tfrac{1}{2k},x')\goto R(x_1)\begin{pmatrix}[0.7]
\sst 0&
\sst 0&
\sst 0&
\sst 0\\
\sst (\zf^1)'&
\sst (\zf^2)'&
\sst (\zf^3)'&
\sst (\zf^4)'
\end{pmatrix}\text{ in }L^2(\Os;\R^{3\times 4}).
\end{equation}
 Indeed, for $\tc{\bar{x}'} \in \mathcal{L}'^{,\rm ext}$ and $i,j \in \{1,2,3,4\}$ with \tc{$|\zf^i-\zf^j|=1$}  we let $T_\ell=(\frac{\ell}{k}+\frac{1}{2k},\bar{x}')+\mathrm{conv}\{0, -\frac{1}{2k}e_1, \zfb^i, \zfb^j\}$, $\ell = -1,0,\ldots,\tc{kL_k}$, so that $\ybk$ is affine on every $T_\ell$. Also set $T = \bigcup_{\ell} T_\ell$ and let $\chi_k$ be the characteristic function of $T$.  Since $\dk2d\ybk(\bar{x}_1-\tfrac{1}{2k},\cdot)$ is constant on each $\Qb'(\bar{x}') = \bar{x}'+(-\frac{1}{2},\frac{1}{2})^2$ \tc{ and $R^{(k)}$ is independent of $\xb'$}, we have 
\begin{align*}
\MoveEqLeft
\int_{(-\frac{1}{k},L_k+\frac{1}{k})\times \Qb'(\bar{x}')}\big|\big[\dk2d\ybk(\bar{x}_1-\tfrac{1}{2k},\xi')\big]_{\bullet i} - R^{(k)}(x_1)\T{(0,(\zf^i)')}\big|^2\md\xb_1\md\xi' \\ 
&=24 \int_{\Omega^{\rm ext}_k}\chi_k\big|\big(\nabla_k\ybk(x) - R^{(k)}(x_1)\big)\T{(0,(\zf^i)')}\big|^2\md x \to 0 
\end{align*}
as $k\to\infty$ by \eqref{eq:FrMU}. Since $\bar{x}'$ and $i$ were arbitrary, \eqref{eq:d2dykConv} now follows from \eqref{eq:RkRU}. 

Thus in \eqref{eq:dgalar}, we combine \eqref{eq:wConvdRdx} with \eqref{eq:d2dykConv} to obtain the limit
\begin{equation}\label{eq:alarLim}
\d2d\a_{\rm right}-\d2d\a_{\rm left}=\T{R}\frac{\pl R}{\pl \xb_1}\begin{pmatrix}[0.7]
\sst 0&
\sst 0&
\sst 0&
\sst 0\\
\sst (\zf^1)'&
\sst (\zf^2)'&
\sst (\zf^3)'&
\sst (\zf^4)'
\end{pmatrix},
\end{equation}
as $-(\pl_{\xb_1}\T{R})R=\T{R}\pl_{\xb_1}R$.

\noindent 3. Finally we bring all contributions together:
\begin{align*}
\bGk
&=\bGk_{\rm long}+\big(\bGk_{\rm left}\,|\,\bGk_{\rm right}\big)\\
&\rightharpoonup \frac{1}{2}\Bigl(G_1(\xb_1)+\T{R}(\xb_1)\frac{\pl R}{\pl \xb_1}(\xb_1)\T{(0,\bar{x}')}\Bigr)(-1,-1,-1,-1,1,1,1,1)\\
&\quad+\big(\d2d\a_{\rm left}\,|\,\d2d\a_{\rm right}\big).
\end{align*}
To finish the proof, we set $\a:=(\a_{\rm left}+\a_{\rm right})/2$ (restricted to $(0,L)\times\mathcal{L}^{\rm ext}$), and use \eqref{eq:alarLim} as well as
\begin{equation*}
\T{R}\frac{\pl R}{\pl \xb_1}=\begin{pmatrix}[0.7]
\sst 0&
\sst -\kappa_2&
\sst -\kappa_3\\
\sst \kappa_2&
\sst 0&
\sst -\tau\\
\sst \kappa_3&
\sst \tau&
\sst 0
\end{pmatrix}.\qedhere
\end{equation*}
\end{proof}
With the help of Proposition~\ref{Gbar}, the proof of Theorem \ref{GammaU}(i) can now be completed following \cite{FrM02} (see also \cite{MMh4,BS06}). We include the details for convenience of the reader. For  $\vec{y} \in \R^{3\times 8}$ we set $W_{\rm tot}(x',\vec{y}) = W_{\rm cell}(\vec{y})$ if $x' \in S$ and $W_{\rm tot}(x',\vec{y}) = W_{\rm surf}(\mathfrak{t}(x'),\vec{y})$ if $x'\in S^{\rm ext} \setminus S$ so that $Q_{\rm tot}(\xb',\cdot)$ is the quadratic form generated by $\nabla^2 W_{\rm tot}(\xb',\bar{\Id})$. 
Using \eqref{eq:EkU}, the non-negativity of $W_{\rm end}$ and the frame-indifference of $W_{\rm tot}$, we can write
\begin{align}\label{eq:liminfWcellU}
\begin{split}
E^{(k)}(y^{(k)})
&\ge\sum_{\hat{x}\in \hat{\Lambda}_k' \cup \hat{\Lambda}_k'^{,\mathrm{surf}}} W_{\rm tot}\big(\hat{x}',\vec{y}^{\,(k)}(\hat{x})\big)\\
&=k \int_{(0,L_k)\times S^{\rm ext}}  W_{\rm tot}\big(x', \T{R^{(k)}(x_1)} \bar{\nabla}_k \ybk(x) \big) \md x\\
& =k \int_{(0,L_k)\times S^{\rm ext}}  W_{\rm tot}\big(x', \bar{\Id} + \tfrac{1}{k} \bGk(\xb) \big) \md x
\end{split}
\end{align}
We let $\chi_k$ be the characteristic function of $\{|\bGk|\leq\sqrt{k}\} \cap \tc{[(0,L_k)\times S^{\rm ext}]}$ and note that $\chi_k\goto 1$ boundedly in measure on $\Os$. As both $W_{\rm tot}(x', \cdot)$ and $\nabla W_{\rm tot}(x', \cdot)$ vanish at $\bar{\Id}$, a Taylor expansion yields 
\begin{equation*}
\chi_kW_{\rm tot}\big(x', \bar{\Id}+\tfrac{1}{k}\bGk\big)
\geq \frac{1}{2k^2}\chi_k Q_{\rm tot}(x', \bGk)-\chi_k\o\big(\tfrac{1}{k}|\bGk|\big),
\end{equation*}
where $\o(t)=o(t^2)$, $t\goto 0$. We deduce that
\begin{align}\label{eq:chiIntsU}
k E^{(k)}(y^{(k)})
&\ge \frac{1}{2} \int_{\Os}  \chi_k Q_{\rm tot}(x', \bGk) \md x 
- k \int_{\Os} \chi_k|\bGk|^2\frac{\o(\tfrac{1}{k}|\bGk|)}{(\frac{1}{k}|\bGk|)^2} \md x 
\end{align}
We can move $\chi_k$ inside the second argument of $Q_{\rm tot}$. As $Q_{\rm tot}(\xb',\cdot)$ is positive semidefinite, the convergence $\chi_k\bGk\weakto\bar{G}$ thus yields
\begin{equation*}
\liminf_{k\goto\infty} kE^{(k)}(y^{(k)})\geq\frac{1}{2}\int_{\Os}Q_{\rm tot}(\xb',\bar{G})\md\xb
\end{equation*}
if the second term in \eqref{eq:chiIntsU} goes to zero. But that follows from the boundedness of $\bGk$ in $L^2(\Os;\R^{3\times 8})$ and the cut-off by $\chi_k$ forcing $L^\infty$-convergence of the fraction involving $\o$.

We substitute in $Q_{\rm tot}$ the representation of $\bar{G}$. By Proposition \ref{Gbar},
\begin{gather*}
\int_{S^{\rm ext}} Q_{\rm tot}(\xb',\bar{G})\md\xb'=\int\limits_{S^{\rm ext}} Q_{\rm tot}\Bigl(\xb',\frac{1}{2}\Bigl[G_1+\T{R}\frac{\pl R}{\pl\xb_1}\T{(0, \bar{\xb}_2, \bar{\xb}_3)}\Bigr](-1,-1,-1,-1,1,1,1,1)\\
+\underbrace{\frac{1}{4}\T{R}\frac{\pl R}{\pl\xb_1}\begin{pmatrix}[0.7]
\sst 0&
\sst 0&
\sst 0&
\sst 0&
\sst 0&
\sst 0&
\sst 0&
\sst 0\\
\sst 1&
\sst 1&
\sst -1&
\sst -1&
\sst -1&
\sst -1&
\sst 1&
\sst 1\\
\sst 1&
\sst -1&
\sst -1&
\sst 1&
\sst -1&
\sst 1&
\sst 1&
\sst -1
\end{pmatrix}}_{\mathfrak{C}}+(\d2d\a|\d2d\a)\Bigr)\md\xb'.
\end{gather*}

The definition of $Q_{\rm cell}^{\rm rel}$ lets us eliminate $G_1$, which only depends on $\xb_1$, and conclude that
\begin{equation*}
\liminf_{k\goto\infty} kE^{(k)}(y^{(k)})
\geq\frac{1}{2}\int_{\O^\square}Q_{\rm tot}(\xb',\bar{G})\md\xb
\geq \frac{1}{2}\int_0^L Q_{\rm cell}^{\rm rel}\left(\T{R}\frac{\pl R}{\pl\xb_1}\right)\md\xb_1.
\end{equation*}
\begin{rem}
Although in continuum theories with homogeneous materials, it can be proved that (an analogue of) the minimizing $\xb_1$-stretch $g$ in \eqref{eq:minProbU} is 0 \cite{Scardia}, here in ultrathin rods it does not seem so clear how to investigate this question.
\end{rem}

\subsection{Proof of the upper bound}\label{sec:upperU}

\begin{proof}[Proof of Theorem \ref{GammaU}(ii)]
Thanks to the $\Gamma$-liminf inequality, it is enough to show 
$$\limsup_{k\goto\infty} kE^{(k)}(y^{(k)})\leq E_{\rm ult}(\y,d_2,d_3).$$
This trivially holds if $(\yb,d_2,d_3)\not\in\calA$. By contrast, if $(\yb,d_2,d_3)\in\calA$, we first additionally suppose that $\yb\in\calC^3([0,L];\R^3)$, $d_2,d_3\in \calC^2([0,L];\R^3)$. Define the sequence of lattice deformations
\begin{equation*}
\ybk(\xb):=\yb(\xb_1)+\frac{1}{k}\xb_2 d_2(\xb_1)+\frac{1}{k}\xb_3 d_3(\xb_1)+\frac{1}{k}q(\xb_1)+\frac{1}{k^2}\b(\xb),\quad\xb\in\{0,\tfrac{1}{k},\ldots,L_k\}\times\mathcal{L}^{\rm ext},
\end{equation*}
where $\b(\cdot,x')\in\calC^1([0,L];\R^3)$ for each $x'\in\mathcal{L}^{\rm ext}$ and $q\in\calC^2([0,L];\R^3)$ are arbitrary for the time being. We interpolate and extend the sequence $\ybk$ to a piecewise affine mapping on $\Omega^{\rm ext}_k$ as in Section~\ref{sec:inter} (by now applying Lemma~\ref{lemma:ext-general} to $S^{\rm ext}_k$ instead of $S_k$ and \tc{then restricting $\ybk$ to $\Omega^{\rm ext}_k$, as no new external atomic layers are needed}) so that 
\begin{align}\label{eq:ends-ok}
\operatorname{ess\, sup}_{\Omega^{\rm ext}_k} \mathrm{dist}^2(\nabla_k \ybk,\mathrm{SO}(3)) 
&\le\tc{C}\operatorname{ess\, sup}_{(0,L_k)\times \frac{1}{kh_k} S^{\rm ext}_k} \mathrm{dist}^2(\nabla_k \ybk,\mathrm{SO}(3)). 
\end{align}
The rescaled discrete gradient of $\ybk$  is
\begin{align*}
[\bar{\nabla}_k\ybk(\xb)]_{\bullet i}
&=k\Bigl[\yb(\bar{\xb}_1+\tfrac{1}{k}\zf_1^i)-\frac{1}{2}\bigl(\yb(\bar{\xb}_1-\tfrac{1}{2k})+\yb(\bar{\xb}_1+\tfrac{1}{2k})\bigr)\Bigr]\\
&\quad +\sum_{s=2}^3\Bigl[(\bar{\xb}_s+\zf_s^i)d_s(\bar{\xb}_1+\tfrac{1}{k}\zf_1^i)-\frac{1}{2}\bar{\xb}_s\bigl(d_s(\bar{\xb}_1-\tfrac{1}{2k})+d_s(\bar{\xb}_1+\tfrac{1}{2k})\bigr)\Bigr]\\
&\quad +q(\bar{\xb}_1+\tfrac{1}{k}\zf_1^i)-\frac{1}{2}\bigl(q(\bar{\xb}_1-\tfrac{1}{2k})+q(\bar{\xb}_1+\tfrac{1}{2k})\bigr) +\tfrac{1}{k}\bigl(\tilde{\b}(\bar{\xb}+\zfb^i)-\tilde{\b}(\bar{\xb})\bigr),
\end{align*}
where $\tilde{\b}$ denotes the usual piecewise affine interpolation of $\b$. Let $R=(\frac{\pl\yb}{\pl x_1}\, |\, d_2\, |\, d_3)$. As in \eqref{eq:liminfWcellU}, frame-indifference for the energy defined in \eqref{eq:EkU} yields
\begin{align}\label{eq:limsupWcellU}
\begin{split}
E^{(k)}(y^{(k)})
&=k\int_{(0,L_k)\times S^{\rm ext}} W_{\rm tot}\bigl(\xb',\bar{\Id}+\tfrac{1}{k}\bar{F}^{(k)}(x)
\bigr)\md\xb \\
&\quad + \sum_{x\in\{-\frac{1}{2k},L_k+\frac{1}{2k}\}\times\mathcal{L}'^{,\rm ext}} W_{\rm end}\big(\mathfrak{t}_k(kx_1,x'),\bar{\nabla}_k\ybk(\xb)\big),
\end{split}
\end{align}
where 
\begin{equation*}
\bar{F}^{(k)}(\xb)=\frac{\T{R(\bar{\xb}_1)}\bar{\nabla}_k\ybk(\xb)-\bar{\Id}}{1/k}.
\end{equation*}

We would like to find the limits of $\bar{F}^{(k)}$ and $\frac{1}{k}\bar{F}^{(k)}$ so that we can let $k\goto\infty$ in \eqref{eq:limsupWcellU}. Fix $i\in\{1,2,\dots,8\}$. For $\xb'\in S^{\rm ext}$ we denote by $\bar{\xb}'$ an element of $\mathcal{L}'^{,\rm ext}$ that is closest to $\xb$. 
Taylor expanding the functions $d_2,d_3,q\in \calC^2([0,L];\R^3)$ about $\bar{\xb}_1$ we deduce that
\begin{align*}
k\Bigl[(\bar{\xb}_s+\zf_s^i)d_s(\bar{\xb}_1+\tfrac{1}{k}\zf_1^i)-\frac{\bar{\xb}_s}{2}\bigl(d_s(\bar{\xb}_1-\tfrac{1}{2k})+d_s(\bar{\xb}_1+\tfrac{1}{2k})\bigr)-d_s(\bar{\xb}_1)\zf_s^i\Bigr]
&\goto (\bar{\xb}_s+\zf_s^i)\frac{\pl d_s}{\pl \xb_1}(\xb_1)\zf_1^i, \\ 
k\Bigl[q(\bar{\xb}_1+\tfrac{1}{k}\zf_1^i)-\frac{1}{2}\bigl(q(\bar{\xb}_1-\tfrac{1}{2k})+q(\bar{\xb}_1+\tfrac{1}{2k})\bigr)\Bigr]
&\goto \zf_1^i\frac{\pl q}{\pl\xb_1}(\xb_1), 
\end{align*}
$s=2,3$, uniformly in $\xb\in\Os$. Similarly, we get by the $\calC^3$-regularity of $\yb$
\begin{align*}
k^2\Bigl[\yb(\bar{\xb}_1+\tfrac{1}{k}\zf_1^i)-\frac{1}{2}(\yb(\bar{\xb}_1-\tfrac{1}{2k})+\yb(\bar{\xb}_1+\tfrac{1}{2k}))\Bigr]-\tc{k}\frac{\pl \yb}{\pl \xb_1}(\bar{\xb}_1)\zf_1^i
&\goto\Bigl(\frac{1}{2}(\zf_1^i)^2-\frac{1}{8}\Bigr)\frac{\pl^2\yb}{\pl\xb_1^2}(\xb_1)=0 
\end{align*}
uniformly in $\xb\in\Os$. Finally, the function $\b$, being uniformly continuous, satisfies 
\begin{align*}
\tilde{\b}(\bar{\xb}+\zfb^i)-\tilde{\b}(\bar{\xb})
&\goto\bigl[\d2d\b(\xb)\,|\,\d2d\b(\xb)\bigr]_{\bullet i}, 
\end{align*}
uniformly in $\xb\in\Os$. Summing up gives 
\begin{align}\label{eq:limsupConvU}
\begin{split}
\MoveEqLeft 
k\Bigl[\bigl[\bar{\nabla}_k\ybk(\xb)\bigr]_{\bullet i}-\bigl(\tfrac{\pl\yb}{\pl\xb_1}\, \big|\, d_2\, \big|\, d_3\bigr)(\bar{\xb}_1)\zf^i\Bigr] \\
&\goto\sum_{s=2}^3(\bar{\xb}_s+\zf_s^i)\frac{\pl d_s}{\pl\xb_1}\zf_1^i+\frac{\pl q}{\pl\xb_1}\zf_1^i 
+[\d2d\b(\xb)\,|\,\d2d\b(\xb)]_{\bullet i}
\end{split}
\end{align}
for any $i\in\{1,2,\dots,8\}$ and so 
\begin{align*}
\bar{F}^{(k)}(x)
&\goto \T{R}(x_1) \Bigl(\frac{\pl R}{\pl\xb_1}(x_1)\T{(0, \bar{\xb}_2, \bar{\xb}_3)} + \frac{\pl q}{\pl\xb_1}(x_1) \Bigr) \T{e_1}\bar{\Id}  \\ 
&\quad + \T{R}(x_1) \frac{\pl R}{\pl\xb_1}(x_1) \Bigl[\zf_1^i\T{(0, \zf_2^i,\zf_3^i)} \Bigr]_{i=1}^8 
+ \T{R}(x_1)\bigl(\d2d\b(\xb)|\d2d\b(\xb)\bigr)
\end{align*}
uniformly in $\xb\in\Os$.  

We first note that by $W_{\rm end}(\mathfrak{t},\cdot) \le C {\rm dist}^2(\cdot,\bar{\rm SO}(3))$, \eqref{eq:limsupConvU} and \eqref{eq:ends-ok}, 
\begin{align*}
\sum_{\xb\in\{-\frac{1}{2k},L_k+\frac{1}{2k}\}\times\mathcal{L}'^{,\rm ext}}
W_{\rm end}\big(\mathfrak{t}_k(kx_1,x'),\bar{\nabla}_k\ybk(\xb)\big) 
\le \frac{C}{k^2},  
\end{align*}
so that this term can be neglected in what follows.
Now Taylor's approximation in \eqref{eq:limsupWcellU} gives 
\begin{align}\label{eq:limsupPartConclU}\begin{split}
\tc{k}E^{(k)}(y^{(k)})
 &\goto\frac{1}{2}\int_{\Os} Q_{\rm tot}\biggl(\xb', \T{R}(x_1) \Bigl(\frac{\pl R}{\pl\xb_1}(x_1)\T{(0, \bar{\xb}_2, \bar{\xb}_3)} + \frac{\pl q}{\pl\xb_1}(x_1) \Bigr) \T{e_1}\bar{\Id}  \\ 
&\quad + \T{R}(x_1) \frac{\pl R}{\pl\xb_1}(x_1) \bigl[\zf_1^i\T{(0, \zf_2^i,\zf_3^i)} \bigr]_{i=1}^8
+ \T{R}(x_1)\bigl(\d2d\b(\xb)|\d2d\b(\xb)\bigr)\biggr)\md\xb.
\end{split}
\end{align}

Next we turn to the case that $(\yb,d_2,d_3)\in\calA$, but $R=(\pl_{\xb_1}\yb|d_2|d_3)$ only belongs to $H^1((0,L);\R^{3\times 3})$. Approximation will allow us to build upon the already finished part of the proof. As this can be done \tc{in analogy to} \cite{MMh4} we only indicate the main steps.

Let $(\a(\xb_1,\cdot),g(\xb_1))$ be a solution of the minimizing problem in the definition of $Q_{\rm cell}^{\rm rel}$. Recall that $\a\in L^2((0,L)\times\mathcal{L}^{\rm ext};\R^3)$, $g\in L^2((0,L);\R^3)$. Find approximating sequences $(\a^{(j)})\subset\calC^1([0,L]\times\mathcal{L}^{\rm ext};\R^3)$, $(R^{(j)})\subset\calC^2([0,L];{\rm SO}(3))$, $(g^{(j)})\subset\calC^2([0,L];\R^3)$ such that $\a^{(j)}\goto\a$ in $L^2((0,L)\times\mathcal{L}^{\rm ext};\R^3)$, $g^{(j)}\goto g$ in $L^2((0,L);\R^3)$ and $R^{(j)}\goto R$ in $H^1([0,L];\R^{3\times 3})$. Further, write $R^{(j)}=(\pl_{\xb_1}\yb^{(j)}|d_2^{(j)}|d_3^{(j)})$ with $d_2^{(j)},d_3^{(j)}\in\calC^2([0,L];\R^3)$ and $\yb^{(j)}$ belonging to $\calC^3([0,L];\R^3)$ such that $\yb^{(j)}(0)=\yb(0)$; this gives $(\yb^{(j)}|d_2^{(j)}|d_3^{(j)})\in\calA$.

For every $j\in\N$, $\b:=R^{(j)}\a^{(j)}$ and $\pl_{\xb_1}q:=R^{(j)}g^{(j)}$ we can construct, by the first part of the proof, $(\yb^{(k,j)})_{k=1}^\infty$ such that
$\yb^{(k,j)}\goto\yb^{(j)}$ in $L^2(\Os;\R^3)$ and $k\frac{\pl \yb^{(k,j)}}{\pl\xb_s}\goto d_s^{(j)}$ in $L^2(\Os;\R^3)$, $s=2,3$ as $k\goto\infty$, so that \eqref{eq:limsupPartConclU} holds with $y^{(k)}$, $d_2$, $d_3$, \tc{$\b$ and $\pl_{\xb_1}q$} replaced with $y^{(k,j)}$, $d_2^{(j)}$, $d_3^{(j)}$, $R^{(j)}\a^{(j)}$ and $R^{(j)}g^{(j)}$, respectively. Finally, diagonalize (take $\yb^{(k,j_k)}$ for a suitable sequence $(j_k)_{k=1}^\infty$) and the proof is finished, since the integral in \eqref{eq:limsupPartConclU} behaves continuously in $R$, $\b$ and $\pl_{\xb_1}q$ with respect to the required topologies.
\end{proof}

\section{Resulting theory for thin rods}\label{sec:T}

We now consider the situation of `thin rods' when the cross section of the rod is not given by a fixed \tc{2D} lattice $\mathcal{L}$ but rather by a macroscopic set $hS \subset \R^2$ whose diameter $h=h_k$ satisfies $\frac{1}{k} \ll h \ll 1$ so that $\Omega = (0, L) \times hS$ is eventually filled with atoms. Again, $L>0$ stands for the rod's length and the cross section is defined in terms of $S$ and $S_k$ as described in Section~\ref{sec:DiscreteModel}. For convenience we also assume that $|S|=1$ and that the axes are oriented in such a way that
\begin{equation}\label{eq:coord}
\int_S\x_2\x_3\md\x'=\int_S\x_2\md\x'=\int_S\x_3\md\x'=0. 
\end{equation}
Since $S$ has a Lipschitz boundary, we can fix $m \ge 1$ such that $S^{\rm ext}_k \supset kh_k S$ for all $k$. 

\subsection{Gamma-convergence}

As in Section \ref{sec:GammaU} in view of Theorem~\ref{MMcomp} and \eqref{rig} the convergence in Theorem~\ref{GammaT} is stated in terms of the piecewise affine interpolations $\yk$ and their rescaled gradients $\nabla_k \yk$ on $\O=(0,L)\times S$ and it suffices to consider limiting configurations $\yb\in H^1((0,L);\R^3)$ and $\tc{d_2,d_3} \in L^2((0,L);\R^3)$.  We remark that, by Theorem~\ref{MMcomp}, \tc{one could equivalently consider the $\Gamma$-limit} in the $L^2_{\rm loc}(\O)$ topology. Also, the convergence could be alternatively formulated in terms of $L^2$ convergence of piecewise constant interpolations of $\yk |_{\tilde{\Lambda}_k}$ and the piecewise constant $\bar{\nabla}_{k} \yk$ to $\y$ and $R\,\bar{\Id}$, respectively; \tc{see} \cite{BrS19}.

\begin{thm}\label{GammaT}
If $k\goto\infty$ and $h_k\goto 0+$ with $kh_k\goto\infty$, the functionals $\frac{1}{k^3h_k^4}E^{(k)}$ $\Gamma$-converge to the functional $E_{\rm th}$ defined below, in the following sense:
\begin{enumerate}
\item[(i)] Let $(y^{(k)})_{k=1}^\infty$ be a sequence of lattice deformations such that their piecewise affine interpolated extensions $(\yk)_{k=1}^\infty$, defined in Section \ref{sec:not}, converge to $\y\in H^1((0,L);\R^3)$ in $L^2(\O;\R^3)$. Let us also assume that $\frac{1}{h_k}\pl_{\x_s}\yk\goto d_s \in L^2((0,L);\R^3)$ in $L^2(\O;\R^3)$, $s=2,3$. Then
$$E_{\rm th}(\y,d_2,d_3)\leq \liminf_{k\goto\infty} \frac{1}{k^3h_k^4}E^{(k)}(y^{(k)}).$$
\item[(ii)] For every $\y\in H^1((0,L);\R^3)$, $d_2,d_3\in L^2((0,L);\R^3)$ there is a sequence of lattice deformations $(y^{(k)})_{k=1}^\infty$ such that their piecewise affine interpolated extensions $(\yk)_{k=1}^\infty$, defined in Section \ref{sec:not}, satisfy $\yk\goto\y$ in $L^2(\O;\R^3)$, $\frac{1}{h_k}\frac{\pl\yk}{\pl\x_s}\goto d_s$ in $L_{\rm loc}^2(\O;\R^3)$ for $s=2,3$, and
$$\lim_{k\goto\infty} \frac{1}{k^3h_k^4}E^{(k)}(y^{(k)})=E_{\rm th}(\y,d_2,d_3).$$
\end{enumerate}
The limit energy functional is given by
\begin{equation*}
E_{\rm th}(\y,d_2,d_3)=\begin{cases}
\frac{1}{2}\int_0^L Q_{\rm cell}^{\rm rel}(\T{R}\pl_{\x_1}R)\md\x_1 & \text{if } (\y,d_2,d_3)\in\calA,\\
+\infty & \text{otherwise},
\end{cases}
\end{equation*}
where $R:=(\pl_{\x_1}\y|d_2|d_3)$ and the class $\calA$ of admissible deformations is as in Theorem \ref{GammaU}. The relaxed quadratic form $Q_{\rm cell}^{\rm rel}:\R^{3\times 3}_{\rm skew}\goto[0,+\infty)$ is defined as
\begin{equation}\label{eq:minProbT}
Q_{\rm cell}^{\rm rel}(A):=\min_{\a\in H^1(S;\R^3)}\int_S Q_{\rm cell}\Bigl(\bigl(A\T{(0, \x_2, \x_3)}\big|\tfrac{\pl\a}{\pl\x_2}\big|\tfrac{\pl\a}{\pl\x_3}\bigr)\bar{\Id}\Bigr)\md\x'.\end{equation}
\end{thm}
\begin{rem}
Theorem \ref{GammaT} is in direct correspondence with the $\Gamma$-limit in \cite{MMh4}. In fact, the work \cite{Conti} shows that for $W_{\rm cell}$ admissible as \textit{full} in Definition \ref{admCellF} and boundary conditions close to a rigid motion, defining the 3D continuum stored energy density $\mathcal{W}$ as $\mathcal{W}(F)=W_{\rm cell}(F\bar{\Id})$, $F\in\R^{3\times 3}$, is justified (the Cauchy--Born rule is valid). If $\mathcal{W}$ is defined this way, then $A:\pl_{F}^2\mathcal{W}(\Id):A=Q_{\rm cell}(A\bar{\Id}):=Q_3(A)$ and with $Q_2$ derived from $Q_3$ by the auxiliary minimization (3.1) in \cite{MMh4}, we get the bending-torsion functional from \cite{MMh4}, since $Q_{\rm cell}^{\rm rel}(A)=Q_2(A)$.
\end{rem}
\begin{rem}
Like in \cite{MMh4}, it can be proved that a solution to the minimum problem in \eqref{eq:minProbT} exists. Since all skew-symmetric matrices are in the kernel of $F\mapsto\nabla^2 W_{\rm cell}(\bar{\Id}):F\bar{\Id}$, we can replace the components of $\nabla\a$ by the components of $\frac{1}{2}(\nabla\a+\T{\nabla}\a)$ on the right-hand side of \eqref{eq:minProbT} and get $H^1$-bounds by a version of Korn's inequality \cite[Theorem~2.5]{Olei} if we require that $\a$ belong to the class
$$\mathcal{V}=\{\beta\in H^1(S;\R^3);\;\int_S\beta\md\x'=0,\;\int_S\nabla\beta\md\x'=0\}.$$
Since $Q_{\rm cell}$ is convex, we obtain the existence of a minimizer $\a$ by the direct method of the calculus of variations. Strict convexity of $Q_{\rm cell}$ on $\R_{\rm sym}^{3\times 3}\bar{\Id}$ implies that the minimizer is unique in $\mathcal{V}$. Further, by analyzing the Euler-Lagrange equations as in \cite{MMequi} we see that $\a$ depends linearly on the entries of $A$ so that $Q_{\rm cell}^{\rm rel}$ is a quadratic form and if $\a(\x_1,\cdot)$ is the solution of \eqref{eq:minProbT} in $\mathcal{V}$ and $A:=[\T{R}R'](\x_1)$, $\x_1\in[0,L]$, we get $\a\in L^2(\O;\R^3)$ and $\pl_{\x_s}\a\in L^2(\O;\R^3)$, $s=2,3$, thanks to $\T{R}R'\in L^2([0,L];\R^{3\times 3})$.
\end{rem}
\tc{\begin{rem}
As mentioned in Example \ref{exampU}, cell energies given by a sum of pairwise interactions may not satisfy assumption (E2) if they do not include an additional penalty term which prevents them from being minimized on improper rotations. Besides, a deficiency of our approach in terms of physical modelling is that interatomic potentials from molecular dynamics are typically bounded near infinity, so the growth assumption (E4) does not apply. However, even energies that are ${\rm O}(3)$-invariant and do not grow quadratically away from $\bar{\rm SO}(3)$ can be treated in case of `sufficiently thin' rods. Following \cite[Section~2.4]{BrS19}, let us suppose that $W_{\rm cell}$ only fulfils (E1), (E3), but also the following alternative assumptions:
\begin{enumerate}
	\item[(E1.1)] $W_{\rm cell}(\vec{y})=W_{\rm cell}(-\vec{y})$ for all $\vec{y}\in V_0^\bot$,
	\item[(E2.1)] $\min_{\vec{y}\in\R^{3\times 8}}W_{\rm cell}(\vec{y})=0$ and $W_{\rm cell}(\vec{y})=\min W_{\rm cell}$ if and only if $\vec{y}=O\bar{\Id}+\vec{c}$ for some $O\in{\rm O}(3)$ and $\vec{c}\in V_0$,
	\item[(E4.1)] there is a constant $\eta>0$ such that $W_{\rm cell}(\vec{y})\geq \eta$ for every $\vec{y}\in V_0^\bot\setminus\mathcal{U}^\pm$, where $\mathcal{U}$ is the neighbourhood of $\bar{\rm SO}(3)$ from (E3) and $\mathcal{U}^\pm:=\mathcal{U}\cup(-\mathcal{U})$.
\end{enumerate}
Since reflections may lead to unnatural folded configurations with zero energy, we also add a nonlocal term to $E^{(k)}$ to avoid colliding atoms, see \cite{BrS19}. Moreover, if we assume that $k^3h_k^4\goto 0+$ (in particular, this also holds in the ultrathin case), then due to our energy scaling, $W_{\rm cell}(\vec{y}^{(k)})$ must be small on every atomic cell. As a result, $W_{\rm cell}$ is never evaluated at points for which a growth assumption would manifest itself. In this setting, Theorem \ref{GammaT} holds with $R^{(k)},R\in{\rm O}(3)$ (analogously in Theorem \ref{GammaU} and up to replacing $\ybk$ with $-\ybk$ in Theorem \ref{MMcomp}).
\end{rem}}

\subsection{Proof of the lower bound}

\tc{To prove Theorem \ref{GammaT}(i), let us assume that $k^{-3}h_k^{-4}E^{(k)}(y^{(k)}) \le C$, whence \eqref{eq:dist-SO} holds due to \eqref{eq:EkU} and \eqref{rig}.} Without loss of generality passing to a \tc{suitable subsequence,} we obtain the piecewise constant $R^{(k)}$ converging to $R(x_1) = (\frac{\partial \yb}{\partial \x_1}(x)\,|\,d_2(x)\,|\,d_3(x))$ as in  Theorem \ref{MMcomp}. From \eqref{eq:FrMU}, for
\begin{equation*}
\Gkn(\x):=\frac{\T{(R^{(k)})}(\x_1)\nabla_{k}\yk(\x)-\Id}{h_k},\quad \x\in \Omega^{\rm ext}_k,
\end{equation*}
we have $\Gkn\weakto G\in L^2(\O;\R^{3\times 3})$ in $L^2(\O;\R^{3\times 3})$, up to a subsequence. Its discrete version
\begin{equation*}
\bGkn(\x):=\frac{k\T{(R^{(k)})}(\x_1)(\yk(\bar{\x}+\zft^i)-\yk(\bar{\x}))_{i=1}^8-\bar{\Id}}{h_k},\quad \x\in \Omega^{\rm ext}_k.
\end{equation*}
is again bounded in $L^2(\tc{\O};\R^{3\times 8})$ (cf.\ \eqref{eq:norms}). Thus $\bGkn\weakto\bar{G}$ in $L^2(\O;\R^{3\times 8})$ for a (not relabelled) subsequence.

The following proposition is contained in \cite{MMh4}.
\begin{prop}\label{prop:GbarT}
Suppose $\Gkn\weakto G\in L^2(\O;\R^{3\times 3})$ in $L^2(\O;\R^{3\times 3})$. Then there are $G_1 \in L^2((0,L);\R^3)$ and $\alpha \in L^2(\O;\R^3)$ with $\nabla'\alpha \in L^2(\O;\R^{3\times 2})$ such that 
\[ G(x) 
   = \Bigl(G_1(\x_1)+\T{R}\frac{\pl R}{\pl \x_1}\T{(0, \x_2, \x_3)}\,\big|\,\nabla'\alpha(x) \Bigr).
\]
\end{prop}
\begin{proof}
See \cite[(3.10) and~(3.13)]{MMh4}.
\end{proof}

Now we explore how the limits $G$ and $\bar{G}$ are connected. Recall the notation $\bar{G}_{\bullet i}$ for the $i$-th column of $\bar{G}$.
\begin{prop}\label{reprBarG}
The representation $\bar{G}_{\bullet i}=G\zf^i$ holds for every $i\in\{1,2,\dots,8\}.$
\end{prop}

\begin{proof} 
Our method is, loosely speaking, to shift everything to a neighbouring lattice block by a direction vector $a$ and handle the resulting remainder. The approach is inspired by \cite{BS06}. Recall that for $x\in \overline{\Omega^{\rm ext}_k}$, we denote by $\bar{x}$ an element of $\Lkexcb$ that is closest to $x$. Take $\bar{\x},\bar{\x}_+=\bar{\x}+a\in\L'^{,\rm ext}_k$, where $a:=\zfb^i-\zfb^j$, and compute
\begin{align}\label{eq:x3shift}\begin{split}
k\T{(R^{(k)})}\bigl(\yk(\bar{\x}+\zft^i)-\yk(\bar{\x})\bigr)-\zf^i
&=k\T{(R^{(k)})}\bigl(\yk(\bar{\x}_++\zft^j)-\yk(\bar{\x}_+)\bigr)-\zf^j\\
&\quad +k\T{(R^{(k)})}\bigl(\yk(\bar{\x}_+)-\yk(\bar{\x})\bigr)+(\zf^j-\zf^i)\end{split}
\end{align}
(note that $\bar{\x}+\zft^i=\bar{\x}_++\zft^j$). Let $\Q=\Q(\bar{\x})=\bar{\x}+(-\frac{1}{2k},\frac{1}{2k})\times(-\frac{1}{2kh_k},\frac{1}{2kh_k})^2$. Property \eqref{eq:surfVolMean} of the piecewise affine interpolation gives
\begin{align}\label{eq:x3diff}
\begin{split}
k(\yk(\bar{\x}_+)-\yk(\bar{\x}))
&=k\dashint_{\Q}\yk(\xi_+)-\yk(\xi)\md\xi\\
&=\dashint_{\Q}\int_0^1k\frac{\md}{\md t}(\yk(\xi+ta))\md t\md\xi
=\dashint_{\Q}\int_0^1 k\nabla\yk(\xi+ta)a\md t\md\xi.
\end{split}
\end{align}
Dividing \eqref{eq:x3shift} by the rod thickness $h_k$ and using \eqref{eq:x3diff}, we derive
\begin{equation}\label{eq:x3preLim}
\bGkn_{\bullet i}=[\bGkn_+]_{\bullet j}+\frac{1}{h_k}\Bigl[\T{(R^{(k)})}\dashint_{\Q}\int_0^1\nabla_{k}\yk(\xi+ta)\md t\md\xi(\zf^i-\zf^j)-(\zf^i-\zf^j)\Bigr],
\end{equation}
where we have set
\begin{equation*}
[\bGkn_+(\x)]_{\bullet j}:=\frac{1}{h_k}[k\T{(R^{(k)}(\x_1))}\bigl(\yk(\bar{\x}_++\zft^j)-\yk(\bar{\x}_+)\bigr)-\zf^j].
\end{equation*}
Fix $\Omega' \subset \subset \Omega$. Let us prove that 
\begin{equation}\label{eq:Gplus}
\bar{G}_+^{(k)}\weakto\bar{G}\text{ in }L^1(\O';\R^{3\times 8}).
\end{equation}
In the first place, shifts by $a$ preserve weak convergence so that in $L^2(\O';\R^{3\times 8})$,
\begin{equation}\label{eq:GplusShift}
\frac{1}{h_k}\Big[k\T{\big(R^{(k)}(\cdot+a_1)\big)}\big(\yk(\cdot+a+\zft^i)-\yk(\cdot+a)\big)_{i=1}^8-\bar{\Id}\Big]\weakto\bar{G}.
\end{equation}
In the second place, due to our construction of $R^{(k)}$ as constant on intervals of length $h_k'$, the difference $R^{(k)}(\cdot+a_1)-R^{(k)}$ can only be nonzero on interfaces of those intervals of constancy. Estimate \eqref{eq:ptCurv} then implies that
\begin{align*}
\int_{\O'}|R^{(k)}(\x_1+a_1)-R^{(k)}(\x_1)|^2\md\x
&\le \frac{|S|}{k} \sum_{i=0}^{\lfloor L_k/h_k'\rfloor-2} 
\bigl|R^{(k)}(ih'_k+\tfrac{3}{2}h'_k)-R^{(k)}(ih'_k+\tfrac{1}{2}h'_k)\bigr|^2 \\ 
&\le \frac{C}{kh_k} \int_{\O^{\rm ext}_k} \mathrm{dist}^2 \bigl( \nabla_k \ybk, \mathrm{SO}(3) \bigr) \md x 
\le \frac{Ch_k}{k}, 
\end{align*}
where the last step \tc{followed from \eqref{eq:dist-SO}}. Thus $\frac{1}{h_k}(R^{(k)}(\cdot+a_1)-R^{(k)})$ tends to $0$ in $L^2$. As $(k(\yk(\cdot+a+\zft^\ell)-\yk(\cdot+a))_{\ell=1}^8)$ is $L^2$-bounded for an analogous reason as $(\bGkn)$, the convergence
$$\frac{1}{h_k}\bigl(\T{R^{(k)}(\cdot+a_1)}-\T{(R^{(k)})}\bigr)k\bigl(\yk(\cdot+a+\zft^\ell)-\yk(\cdot+a)\bigr)_{\ell=1}^8\stackrel{L^1}{\goto} 0$$
combined with \eqref{eq:GplusShift} establishes \eqref{eq:Gplus}. 

By definition, $R^{(k)}$ is constant on $\Q$, so the remainder term in \eqref{eq:x3preLim} equals
\begin{align}
\MoveEqLeft 
\int_0^1\dashint_{\Q}\frac{1}{h_k}\big[\T{(R^{(k)}(\xi_1))}\nabla_{k}\yk(\xi+ta)(\zf^i-\zf^j)-(\zf^i-\zf^j)\big]\md\xi\md t\nonumber\\
&=\dashint_{\Q}\int_0^1\Gkn(\xi+ta)(\zf^i-\zf^j)\md t\md\xi\label{eq:Pkt}\\
&+\int_0^1\dashint_{\Q}\Bigl[\frac{1}{h_k}\Bigl(\T{\bigl(R^{(k)}(\xi_1)-R^{(k)}(\xi_1+ta_1)\bigr)}\nabla_{k}\yk(\xi+ta)(\zf^i-\zf^j)\Bigr)\Bigr]\md t.\label{eq:Rxta}
\end{align}
Term \eqref{eq:Pkt} weakly converges to $G(\zf^i-\zf^j)$ since for any $\varphi\in C^{\infty}_c(\O)$ 
\begin{align*} \int_{\Omega}\dashint_{\Q(\bar{x})}\int_0^1\Gkn(\xi+ta)\md t\md\xi \varphi(x) \md x 
&=\int_{\Omega} \int_0^1\Gkn(\xi+ta)\md t \dashint_{\Q(\bar{\xi})} \varphi(x) \md x  \md \xi \\ 
&\tc{ \goto \int_{\Omega} G(\xi) \varphi(\xi) \md \xi.}
\end{align*}
In \eqref{eq:Rxta}, $\nabla_{k}\yk(\cdot+ta)$ is $L^2$-bounded uniformly in $t$ by \eqref{eq:FrMU} and as above we see that $\frac{1}{h_k}(R^{(k)}(\cdot+\tc{ta_1})-R^{(k)})$ converges to $0$ in $L^2(\O';\R^{3\times 3})$ uniformly in $t$, so the whole term vanishes in the limit.

Thus, passing to the limit in \eqref{eq:x3preLim}, we conclude that
\begin{equation*}
\bar{G}_{\bullet i}-\bar{G}_{\bullet j}=G(\zf^i-\zf^j).
\end{equation*}
The assertion now follows by summing over $j$ and using that $\sum_{j=1}^8 \bar{G}_{\bullet j} = \sum_{j=1}^8 \zf^j = 0$. 
\end{proof}

We can now finish the proof of Theorem \ref{GammaT}(i). 

By \eqref{eq:EkU}, the non-negativity of \tc{$W_{\rm surf}$ and $W_{\rm end}$,} and the frame-indifference of $W_{\rm cell}$, we estimate
as in \eqref{eq:liminfWcellU}
\begin{equation*}
E^{(k)}(y^{(k)})\ge\sum_{\hat{x}\in \hat{\Lambda}_k'} W_{\rm cell}\big(\vec{y}^{\,(k)}(\hat{x})\big)
=k^3h_k^2 \int_{\O} \chi_k(x) W_{\rm cell}\big(\bar{\Id} + h_k \bGk(\xb) \big) \md x,
\end{equation*}
where now $\chi_k$ is the characteristic function of $\{|\bGk|\leq\sqrt{1/h_k}\} \cap [(0,L_k)\times \frac{1}{kh_k}S_k]$. The same arguments as in the ultrathin case, cf.\ also \cite{FrM02,MMh4,BS06}, lead to 
\begin{equation*}
\liminf_{k\goto\infty} \frac{1}{k^3h_k^4}E^{(k)}(y^{(k)})
\geq\int_{\O}\frac{1}{2}Q_{\rm cell}(\bar{G})\md\x.
\end{equation*}

By Proposition \ref{reprBarG}, $Q_{\rm cell}(\bar{G})=Q_{\rm cell}(G\bar{\Id})$. The proof can be \tc{completed} as in \cite{MMh4}: Setting $c_s(x_1) = \int_S \frac{\pl\a}{\pl\x_s}\md x'$, $s=2,3$, and invoking Proposition~\ref{prop:GbarT} \tc{and \eqref{eq:coord}} we find 
\begin{align*}
\int_S Q_{\rm cell}(\bar{G})\md\x'
&=\int_S Q_{\rm cell}\bigl((G_1(\x_1)\,\big|\,c_2(x_1)\big|c_3(x_1))\bar{\Id}\bigr)\md\x'\\
&\quad+\int_S Q_{\rm cell}\Bigl(\Bigl(\T{R}\frac{\pl R}{\pl \x_1}\T{(0, \x_2, \x_3)}
\Big|\frac{\pl\bar{\a}}{\pl\x_2}\Big|\frac{\pl\bar{\a}}{\pl\x_3}\Bigr)\bar{\Id}\Bigr)\md\x',
\end{align*}
\tc{where $\bar{\a}(\x):=\a(\x)-\x_2c_2(\x_1)-\x_3c_3(\x_1)$}. Thus the definition of $Q_{\rm cell}^{\rm rel}$ lets us conclude that
\begin{equation*}
\liminf_{k\goto\infty} \frac{1}{k^3h_k^4}E^{(k)}(y^{(k)})\geq\int_\O\frac{1}{2}Q_{\rm cell}(\bar{G})\md\x\geq 0+\frac{1}{2}\int_0^L Q_{\rm cell}^{\rm rel}\left(\T{R}\frac{\pl R}{\pl\x_1}\right)\md\x_1. 
\end{equation*}

\subsection{Proof of the upper bound}\label{sec:upperT}

\begin{proof}[Proof of Theorem \ref{GammaT}(ii)]
In the nontrivial case that $(\y,d_2,d_3)\in\calA$, we first additionally suppose that $\y\in\calC^3([0,L];\R^3)$, $d_2,d_3\in \calC^2([0,L];\R^3)$. For $\b\in\calC^1(\R^3;\R^3)$ to be fixed later, define the sequence 
\begin{equation*}
\yk(\x):=\y(\x_1)+h_k\x_2 d_2(\x_1)+h_k\x_3 d_3(\x_1)+h_k^2\b(\x),\quad\x\in \{0,\tfrac{1}{k},\ldots,L_k\}\times\frac{1}{kh_k}\mathcal{L}^{\rm ext}_k.
\end{equation*}
We extend and interpolate the sequence $\yk$ on $\Omega^{\rm ext}_{k}$ in the same way as in Section \ref{sec:upperU} so that, in particular, \eqref{eq:ends-ok} holds true again. The rescaled discrete gradient $\bar{\nabla}_{k}\yk(\x)=k[\yk(\bar{\x}+\zft^i)-\frac{1}{8}\sum_{j=1}^8 \yk(\bar{\x}+\zft^j)]_{i=1}^8$ of $\yk$ reads
\begin{align*}
[\bar{\nabla}_k\ybk(\xb)]_{\bullet i}
&=k\Bigl[\yb(\bar{\xb}_1+\tfrac{1}{k}\zf_1^i)-\frac{1}{2}\bigl(\yb(\bar{\xb}_1-\tfrac{1}{2k})+\yb(\bar{\xb}_1+\tfrac{1}{2k})\bigr)\Bigr]\\
&\quad +\sum_{s=2}^3kh_k\Bigl[(\bar{\xb}_s+\zfb_s^i)d_s(\bar{\xb}_1+\tfrac{1}{k}\zf_1^i)-\frac{1}{2}\bar{\xb}_s\bigl(d_s(\bar{\xb}_1-\tfrac{1}{2k})+d_s(\bar{\xb}_1+\tfrac{1}{2k})\bigr)\Bigr]\\
&\quad +kh_k^2\bigl(\tilde{\b}(\bar{\xb}+\zfb^i)-\tilde{\b}(\bar{\xb})\bigr),
\end{align*}
where $\tilde{\b}$ denotes the piecewise affine discretization of \tc{$\b|_{\Lkexb}$} described in Section \ref{sec:not}.

As in \eqref{eq:limsupWcellU} we obtain 
\tc{\begin{gather}\label{eq:limsupWcellT}
\begin{split}
E^{(k)}(y^{(k)})
=k^3h_k^2\int_{(0,L_k)\times S_k} W_{\rm cell}\bigl(\xb',\bar{\Id}+h_k\bar{F}^{(k)}(x)
\bigr)\md\xb+\sum_{\hat{x}\in\hat{\Lambda}_k'^{,\mathrm{surf}}} W_{\rm surf}\bigl(\mathfrak{t}_k(\hat{x}'),\bar{\nabla}\hat{y}^{(k)}(\hat{x})\bigr) \\
+ \sum_{x\in\{-\frac{1}{2k},L_k+\frac{1}{2k}\}\times\mathcal{L}'^{,\rm ext}_k} W_{\rm end}\big(\mathfrak{t}_k(kx_1,x'),\bar{\nabla}_k\ybk(\xb)\big),
\end{split}
\end{gather}}
where 
\begin{equation*}
\bar{F}^{(k)}(\xb)
=\frac{\T{R}(\bar{\xb}_1)\bar{\nabla}_k\ybk(\xb)-\bar{\Id}}{h_k}.
\end{equation*}

Fixing $i\in\{1,2,\dots,8\}$, we deduce the following convergences, analogous to their counterparts in ultrathin rods, uniformly in $\x\in[0,L]\times S'$ for each bounded domain $S' \supset\supset S$: 
\begin{align*}
kh_k\bigl(\tilde{\b}(\bar{\x}+\zfb^i)-\tilde{\b}(\bar{\x})\bigr)&\goto\frac{\pl\b}{\pl\x_2}(\x)\zf_2^i+\frac{\pl\b}{\pl\x_3}(\x)\zf_3^i
\end{align*}
and 
\begin{align*}
k\Bigl[(\bar{\x}_s+\zfb_s^i)d_s(\bar{\x}_1+\tfrac{1}{k}\zf_1^i)-\frac{1}{2}\bar{\x}_s\bigl(d_s(\bar{\x}_1-\tfrac{1}{2k})+d_s(\bar{\x}_1+\tfrac{1}{2k})\bigr)\Bigr]
-\frac{1}{h_k}d_s(\bar{\x}_1)\zf_s^i \goto \x_s\frac{\pl d_s}{\pl \x_1}(\x_1)\zf_1^i
\end{align*}
for $s=2,3$, as well as 
\begin{align*}
\frac{k}{h_k}\Bigl[\y(\bar{\x}_1+\tfrac{1}{k}\zf_1^i)-\frac{1}{2}\bigl(\y(\bar{\x}_1-\tfrac{1}{2k})+\y(\bar{\x}_1+\tfrac{1}{2k})\bigr)\Bigr]-\frac{1}{h_k}\frac{\pl \y}{\pl \x_1}(\bar{\x}_1)\zf_1^i\goto 0. 
\end{align*}
Summing them up leads to
\begin{align}\label{eq:limsupConvT}
\frac{1}{h_k}\Bigl[\bigl[\bar{\nabla}_{k}\yk(\x)\bigr]_{\bullet i}-\bigl(\tfrac{\pl\y}{\pl\x_1}\,\big|\,d_2\,\big|\, d_3\bigr)(\bar{\x}_1)\zf^i\Bigr]&\goto\sum_{s=2}^3\x_s\frac{\pl d_s}{\pl\x_1}\zf_1^i+\frac{\pl\b}{\pl\x_s}\zf_s^i
\end{align}
for any $i\in\{1,2,\dots,8\}$. 

Now we first notice that $\max_{\mathfrak{t} \in \mathfrak{T}} \big( W_{\rm surf}(\mathfrak{t},\cdot) + W_{\rm end}(\mathfrak{t},\cdot) \big) \le C {\rm dist}^2(\cdot,\bar{\rm SO}(3))$, \eqref{eq:limsupConvT}, \eqref{eq:ends-ok}, and the estimate $\sharp (\Lkex\setminus\Lambda_k) \le C(k^2h_k+k^2h_k^2) \le Ck^2h_k$ give 
\begin{equation*}\sum_{\hat{x}\in\hat{\Lambda}_k'^{,\mathrm{surf}}} W_{\rm surf}\bigl(\mathfrak{t}_k(\hat{x}'),\bar{\nabla} y^{(k)}(\hat{x})\bigr)
+\sum_{\hat{x}\in\{-\frac{1}{2},kL_k+\frac{1}{2}\}\times\mathcal{L}'^{,\rm ext}_k} W_{\rm end}\bigl(\mathfrak{t}_k(\hat{x}),\bar{\nabla} y^{(k)}(\hat{x})\bigr) 
\leq Ck^2 h_k^3. 
\end{equation*} 
Hence, Taylor's approximation in \eqref{eq:limsupWcellT} yields
\tc{\begin{equation*}
\frac{1}{k^3h_k^4}E^{(k)}(y^{(k)})\goto\frac{1}{2}\int_{\O} Q_{\rm cell}\Bigl(\T{R}\bigl(\x_2\tfrac{\pl d_2}{\pl\x_1}+\x_3\tfrac{\pl d_3}{\pl\x_1}\,\big|\,\tfrac{\pl\b}{\pl\x_2}\,\big|\,\tfrac{\pl\b}{\pl\x_3}\bigr)\bar{\Id}\Bigr)\md\x.
\end{equation*}}
For general $(\y,d_2,d_3)\in\calA$ with $R=(\pl_{\x_1}\y|d_2|d_3)\in H^1((0,L);\R^{3\times 3})$ we may proceed by approximation exactly as in Section~\ref{sec:upperU}, now using that the solution $\a(\x_1,\cdot)$ of the minimizing problem in the definition of $Q_{\rm cell}^{\rm rel}$ is such that $\a\in L^2(\O;\R^3)$ and $\pl_{\x_s}\a\in L^2(\O;\R^3)$, $s=2,3$.
\end{proof}

\section*{Acknowledgements}
The authors acknowledge the support of the Deutsche Forschungsgemeinschaft (DFG, German Research Foundation) within the Priority Programme SPP 2256 `Variational Methods for Predicting Complex Phenomena in Engineering Structures and Materials'.

\bibliographystyle{alpha} 
\renewcommand{\bibname}{References}
\bibliography{bibliography}
\end{document}